\documentclass[11pt, oneside]{amsart}

\usepackage{graphicx} %
\usepackage{amsmath, amsthm, amssymb}
\usepackage{hyperref}
\usepackage[capitalize]{cleveref}
\usepackage{xspace}
\usepackage{tikz-cd}
\usepackage{tikz}
\usepackage[colorinlistoftodos, textsize=small]{todonotes}
\usepackage{enumitem}
\usepackage{xparse}
\usepackage{mathtools}
\usepackage{etoolbox}
\usetikzlibrary{calc}
\usepackage{pgf}
\usepackage{dsfont}
\usepackage{kantlipsum} %
\usepackage{subcaption} %

\usepackage{tabularray}
\UseTblrLibrary{diagbox}
\NewColumnType{C}{Q[c,$]}

\calclayout

\newcommand{\longhookrightarrow}{\lhook\joinrel\longrightarrow}
\newcommand{\NN}{\ensuremath{\mathbb N}}
\newcommand{\RR}{\ensuremath{\mathbb R}}
\newcommand{\ZZ}{\ensuremath{\mathbb Z}}
\newcommand{\CC}{\ensuremath{\mathbb C}}
\newcommand{\C}{{\mathbb{C}}}
\renewcommand{\Re}{\operatorname{Re}}
\renewcommand{\Im}{\operatorname{Im}}
\newcommand{\x}{x}
\newcommand{\y}{y}
\newcommand{\s}{s}

\newcommand{\tripr}{ F_2^{(a)} \times F_2^{(b)} \times F_2^{(c)} }
\newcommand{\CZ}{\CC\setminus\{0\}}
\newcommand{\inv}[1]{\overline{#1}}
\newcommand{\abs}[1]{\left|#1\right|}
\newcommand{\SB}[1][3]{\mathrm{SB}(#1)}
\newcommand{\SBgens}{x_1,x_2,y_1,y_2,s}

\newcommand{\norm}[1]{\|#1\|_\infty}

\newcommand{\PB}{\mathrm{PB}}
\renewcommand{\epsilon}{\varepsilon}

\NewDocumentCommand{\Dehn}{m O{N}}{\delta_{#1}(#2)}
\NewDocumentCommand{\radius}{m O{N}}{\rho_{#1}(#2)}
\let\asympleq\preccurlyeq
\let\asympgeq\succcurlyeq
\newcommand{\disonwords}{F_4} %
\newcommand{\trivialwords}{W} %

\NewDocumentCommand{\freecopy}{m m}{F_{#1}^{(#2)}}
\NewDocumentCommand{\fpgens}{o O{\r}}{\mathcal A_{#2}\IfValueT{#1}{^{(#1)}}} %
\NewDocumentCommand{\fpgen}{m m}{a_{#1}^{(#2)}}
\NewDocumentCommand{\invfpgen}{m m}{\inv a_{#1}^{(#2)}}
\NewDocumentCommand{\invfpgens}{o O{\r}}{\inv{\mathcal A}_{#2} \IfValueT{#1}{^{(#1)}}} %

\NewDocumentCommand{\fprels}{}{\mathcal C}
\NewDocumentCommand{\kgens}{o O{\r}}{\mathcal X_{#2}\IfValueT{#1}{^{(#1)}}} %
\NewDocumentCommand{\kgen}{m m}{x_{#1}^{(#2)}}
\NewDocumentCommand{\invkgen}{m m}{\inv x_{#1}^{(#2)}}
\NewDocumentCommand{\invkgens}{o O{\r}}{\inv{\mathcal X}_#2\IfValueT{#1}{^{(#1)}}} %
\newcommand{\diaggens}{\Delta}
\NewDocumentCommand{\normal}{o m}{%
	w_{#2}%
	\IfValueT{#1}{%
		\ifthenelse{\equal{#1}{d}}{%
			^\diaggens%
		}{%
			^{(#1)}%
		}%
	}
}%

\NewDocumentCommand{\afgen}{m}{\xi_{#1}}

\NewDocumentCommand{\kalph}{m}{\kgens[#1]}

\let\diagalph\diaggens

\def\r{r}
\def\n{n}
\NewDocumentCommand{\K}{O{\n} O{\r} O{}}{{K^{#1}_{#2}(\ifblank{#3}{#2}{#3})}}
\NewDocumentCommand{\R}{ O{\r}O{\n}}{{{\mathcal R}_{#1}^{#2}}}
\NewDocumentCommand{\Rthick}{O{\bq} O{\n} O{\r}}{{{\mathbf{\mathcal R}}_{#3,#2}^{#1}}}
\NewDocumentCommand{\Aii}{m O{\r}}{{\mathcal{A}^{(#1)}_{#2}}}
\NewDocumentCommand{\A}{O{\n} O{\r}}{{\mathcal{A}_{#2}^{#1}}}
\NewDocumentCommand{\Fii}{m O{\r}}{{F^{(#1)}_{#2}}}
\NewDocumentCommand{\Xii}{m O{\r}}{{\mathcal{X}^{(#1)}_{#2}}}
\NewDocumentCommand{\mXii}{m O{\r}}{{\inv{\mathcal{X}}^{(#1)}_{#2}}}
\NewDocumentCommand{\X}{O{\n} O{\r}}{{\mathcal{X}_{#2}^{#1}}}
\NewDocumentCommand{\unor}{O{\r}}{{\mathds{1}_{#1}}}
\NewDocumentCommand{\xii}{m m}{{x^{(#1)}_{#2}}}

\NewDocumentCommand{\aii}{m m}{{a^{(#1)}_{#2}}}

\NewDocumentCommand{\prii}{m O{\r}}{\operatorname{pr}^{(#1)}_{#2}}
\NewDocumentCommand{\thickener}{m O{\bq}}{{\kappa_#2^#1}}
\NewDocumentCommand{\pres}{ O{r} O{\ell} O{m} O{n}}{\ensuremath{\Gamma_{#2,#3,#4}^{#1}}\xspace}
\NewDocumentCommand{\prestilde}{ O{r} O{\ell} O{m} O{n}}{\ensuremath{\widetilde\Gamma_{#2,#3,#4}^{#1}}\xspace}

\NewDocumentCommand{\dirprod}{O{\r} O{\n}}{%
	\ifthenelse{\equal{#2}{3}}{%
		\freecopy{#1}{a} \times \freecopy{#1}{b} \times \freecopy{#1}{c}%
	}{%
		\freecopy{#1}{1} \times \dots \times \freecopy{#1}{#2}%
	}%
}

\newcommand{\relations}{\mathcal{R}}
\let\length\abs
\newcommand{\ton}{^{(n)}}

\newcommand{\toi}{^{(i)}}

\NewDocumentCommand{\pa}{m O{\bq}}{{\mathbf x}_{#2}^{(#1)}}
\NewDocumentCommand{\mpa}{m O{\bq}}{\overline{\mathbf{x}}_{#2}^{(#1)}}
\NewDocumentCommand{\paro}{m m O{1} O{\bq}}{\pa{#1}[#4_{[#3:#2]}]}
\NewDocumentCommand{\mparo}{m m O{1} O{\bq}}{\mpa{#1}[#4_{[#3:#2]}]}

\NewDocumentCommand{\pax}{O{\bq}}{{\mathbf x}_{#1}}
\NewDocumentCommand{\mpax}{O{\bq}}{\overline{\mathbf{x}}_{#1}}
\NewDocumentCommand{\parox}{m O{1} O{\bq}}{\pax[#3_{[#2:#1]}]}
\NewDocumentCommand{\mparox}{m O{1} O{\bq}}{\mpax[#3_{[#2:#1]}]}

\NewDocumentCommand{\pay}{O{\bq}}{{\mathbf y}_{#1}}
\NewDocumentCommand{\mpay}{O{\bq}}{\overline{\mathbf{y}}_{#1}}
\NewDocumentCommand{\paroy}{m O{1} O{\bq}}{\pay[#3_{[#2:#1]}]}
\NewDocumentCommand{\mparoy}{m O{1} O{\bq}}{\mpay[#3_{[#2:#1]}]}

\DeclareMathOperator{\Conf}{Conf}
\newcommand{\push}{\operatorname{push}}
\DeclareMathOperator{\Area}{Area}

\NewDocumentCommand{\swaprels}{}

\newcommand{\I}{I}

\newcommand{\Cay}[2]{\operatorname{Cay}_{#2}(#1)}

\newcommand{\rar}{\rightarrow}
\newcommand{\ol}[1]{\overline{#1}}

\newcommand{\ppres}[2]{\langle #1\  |\  #2 \rangle}
\newcommand{\bN}{\mathbb{N}}

\newcommand{\cR}{\mathcal{R}}

\newcommand{\cX}{\mathcal{X}}

\newcommand{\bq}{\mathbf{q}}

\newcommand{\fibration}{\psi}
\newcommand{\free}[1]{F (#1)}
\newcommand{\eqfree}{=}
\newcommand{\eqg}[1]{=_{#1}}

\newcommand{\coloneq}{\coloneqq}
\newcommand{\eqcolon}{\eqqcolon}

\NewDocumentCommand{\simplecomms}{ O{\r} }{\mathcal R_{#1,1}}
\NewDocumentCommand{\comms}{ O{\r} }{\mathcal R_{#1,2}}
\NewDocumentCommand{\swaps}{ O{\r} }{\mathcal R_{#1,3}}
\NewDocumentCommand{\triplecomms}{ O{\r} }{\mathcal R_{#1,4}}
\NewDocumentCommand{\quadcomms}{ O{\r} }{\mathcal R_{#1,5}}

\let\originalleft\left
\let\originalright\right
\renewcommand{\left}{\mathopen{}\mathclose\bgroup\originalleft}
\renewcommand{\right}{\aftergroup\egroup\originalright}

\RequirePackage{mathtools}
\DeclarePairedDelimiter\set\{\}

\DeclarePairedDelimiter{\ceil}{\lceil}{\rceil}
\DeclarePairedDelimiterX\commsub[1]{[}{]}{#1,#1}

\NewDocumentCommand{\sequence}{O{1} m o}{#2_#1, \dots \IfValueT{#3}{, #2_#3}} %

\newcommand\generatedby[1]{\presentation{#1}{}}
\DeclarePairedDelimiterX\presentation[2]\langle\rangle{
	#1 \ifblank{#2}{}{\mid #2}
}
\NewDocumentCommand{\finitetype}{o}{\ensuremath{\IfValueTF{#1}{\mathcal F_#1}{\mathcal F}}\xspace}

\NewDocumentCommand{\shortexactsequence}{m O{} m O{} m O{}}{
	\begin{tikzcd}[ampersand replacement=\&, column sep=small]
		1 \arrow[r] \& #1 \arrow[r, "#2"] \& #3 \arrow[r, "#4"] \& #5 \arrow[r] \& 1#6
	\end{tikzcd}
}

\NewDocumentCommand{\range}{O{1} m}{%
	\ifthenelse{\equal{#1}{1} \AND \equal{#2}{3}}{%
		\set{1,2,3}%
	}{%
		\ifthenelse{\equal{#1}{1} \AND \equal{#2}{2}}{%
			\set{1,2}%
		}{%
			\set{#1, \dots, #2}%
		}%
	}
}

\newcommand{\isom}{\cong}

\newcommand{\suchthat}{\mid}
\newcommand{\Suchthat}{\ \middle|\ }

\newcommand\bigfun[5]{%
	\begin{tikzcd}[
			column sep=2em,
			row sep=1ex,
			ampersand replacement=\&
		]
		#1\colon \&[-2.5em]
		#2\vphantom{#3} \arrow[r] \&
		#3\vphantom{#2} \\
		\&
		#4\vphantom{#5}  \arrow[r,mapsto] \&
		#5\vphantom{#4}
	\end{tikzcd}%
}

\newcommand{\smallgroup}{\ensuremath{K_2^3(2)}\xspace}

\newtheorem{lemma}{Lemma}[section]
\newtheorem{proposition}[lemma]{Proposition}
\newtheorem{theorem}[lemma]{Theorem}
\newtheorem{corollary}[lemma]{Corollary}

\theoremstyle{remark}
\newtheorem{remark}[lemma]{Remark}

\theoremstyle{definition}
\newtheorem{definition}[lemma]{Definition}
\newtheorem{question}[lemma]{Question}
\newcounter{MainTheorems}

\newtheorem{maintheorem}[MainTheorems]{Theorem}
\crefname{maintheorem}{Theorem}{Theorems}
\crefname{equation}{}{}
\crefformat{enumi}{#2\textup{(#1)}#3}
\crefmultiformat{enumi}{#2\textup{(#1)}#3}{,~#2\textup{(#1)}#3}{,~#2\textup{(#1)}#3}{ and~#2\textup{(#1)}#3}

\makeatletter
\providecommand\@dotsep{5}
\def\listtodoname{List of Todos}
\def\listoftodos{\@starttoc{tdo}\listtodoname}
\makeatother

\newcommand{\todoerror}[2][]{}
\newcommand{\todowarn}[2][]{}
\newcommand{\todoremark}[2][]{}
\newcommand{\todohint}[2][]{}
\newcommand{\done}[2][]{}
\renewcommand{\todoerror}[2][]{}
\renewcommand{\todowarn}[2][]{}
\renewcommand{\todoremark}[2][]{}
\renewcommand{\todohint}[2][]{}
\renewcommand{\done}[2][]{}

\title[Dehn functions of SPFs]{Dehn functions of subgroups of products of free groups \\ Part II: Precise computations}
\author{Dario Ascari}
\address{\parbox{\linewidth}{Department of Mathematics, University of the Basque Country,\\
		Barrio Sarriena, Leioa, 48940, Spain\vspace{1.5pt}}}
\email{ascari.maths@gmail.com}
\author{Federica Bertolotti}
\address{Scuola Normale Superiore, Piazza dei Cavalieri 7, 56126 Pisa, Italy}
\email{federica.bertolotti@sns.it}
\author{Giovanni Italiano}
\address{\parbox{\linewidth}{Mathematical Institute, University of Oxford, \\Andrew Wiles Building,
		Woodstock Road, OX2 6GG Oxford, UK}\vspace{1.5pt}}
\email{italiano@maths.ox.ac.uk}
\author{Claudio Llosa Isenrich}
\address{Faculty of Mathematics, KIT, Englerstr. 2, 76131 Karlsruhe, Germany}
\email{claudio.llosa@kit.edu}
\author{Matteo Migliorini}
\address{Faculty of Mathematics, KIT, Englerstr. 2, 76131 Karlsruhe, Germany}
\email{matteo.migliorini@kit.edu}

\keywords{Dehn function, direct product of free groups, residually free groups}
\subjclass{20F65 (20F05, 20F67, 20F69, 57M07)}

\begin{document}

\begin{abstract}
    We prove that the Bridson--Dison group has quartic Dehn function, thereby providing the first precise computation of the Dehn function of a subgroup of a direct product of free groups with super-quadratic Dehn function. We also prove that coabelian subgroups of direct products of $n$ free groups of finiteness type $\mathcal{F}_{n-1}$ and of corank $r\leq n-2$ have quadratic Dehn functions.
\end{abstract}

\maketitle

\addtocontents{toc}{\protect\setcounter{tocdepth}{1}}

\section{Introduction}

\newcommand{\comment}[1]{}

  A fundamental geometric invariant of finitely presented groups is their Dehn function. Combinatorially, it quantifies the complexity of solving the word problem, while geometrically it provides an optimal isoperimetric inequality. Here we focus on precise computations of Dehn functions of subgroups of direct products of free groups (short: SPFs), continuing the study initiated in \cite{UniformBounds-25}, where we focused on uniform upper bounds. In particular, we give the first precise computation of a super-quadratic Dehn function of a SPF:

  \begin{maintheorem}\label{main:disons-group}
	  The kernel of the morphism $F_2\times F_2 \times F_2 \to \mathbb{Z}^2$ which is surjective on every factor has quartic Dehn function.
  \end{maintheorem}

  We emphasize that we obtain the lower bound in \cref{main:disons-group} by introducing a new homotopical invariant, based on braid groups. Furthermore, %
  we prove:

  \begin{maintheorem}\label{main:quadratic-case}
	  Let $n,r$ be integers with $n\geq r+2\geq 4$ and let $K$ be an SPF of type $\mathcal{F}_{n-1}$ in a direct product of $n$ which is virtually coabelian of corank $r$. Then $\delta_K(N)\asymp N^2$.

  \end{maintheorem}

\cref{tbl:dehn-functions} gives a a summary of the results of this work and \cite{UniformBounds-25}. We will now explain in detail the context of our results, and the ideas behind their proofs.

\begin{table}
	\begin{tblr}{
			colspec={CCCCCCCC},
			hline{2-8} = {solid},
			vlines={2-7}{solid},
			cell{3}{2-6} = {green},
			cell{4}{4-6} = {green},
			cell{4}{4-6} = {green},
			cell{5}{5-6} = {green},
			cell{6}{6-6} = {green},
			cell{1}{1} = {r=1,c=6}{c,mode=text}
		}
		Best upper bounds                                           \\
		\diagbox{r}{n} & 3      & 4      & 5      & 6      & \dots  \\
		2              & N^4    & N^2    & N^2    & N^2    & \dots  \\
		3              & N^8    & N^5    & N^2    & N^2    & \dots  \\
		4              & N^9    & N^5    & N^4    & N^2    & \dots  \\
		5              & N^9    & N^5    & N^4    & N^4    & \dots  \\
		\vdots         & \vdots & \vdots & \vdots & \vdots & \ddots
	\end{tblr}
	\hspace{0.3cm}
	\begin{tblr}{
			colspec={CCCCCCCC},
			hline{2-8} = {solid},
			vlines={2-7}{solid},
			cell{3}{2-6} = {green},
			cell{4}{4-6} = {green},
			cell{4}{4-6} = {green},
			cell{5}{5-6} = {green},
			cell{6}{6-6} = {green},
			cell{1}{1} = {r=1,c=6}{c,mode=text}
		}
		Best lower bounds                                           \\
		\diagbox{r}{n} & 3      & 4      & 5      & 6      & \dots  \\
		2              & N^4    & N^2    & N^2    & N^2    & \dots  \\
		3              & N^4    & N^2    & N^2    & N^2    & \dots  \\
		4              & N^4    & N^2    & N^2    & N^2    & \dots  \\
		5              & N^4    & N^2    & N^2    & N^2    & \dots  \\
		\vdots         & \vdots & \vdots & \vdots & \vdots & \ddots
	\end{tblr}
	\caption{This table illustrates the best known upper and lower bounds for $ \K $, obtained by combining the results of \cref{main:quadratic-case,main:disons-group} and of our previous paper (see \cite[Table 1]{UniformBounds-25}). In the table we have highlighted the cases where the two bounds coincide and give a precise computation of the Dehn function.}
	\label{tbl:dehn-functions}
\end{table}

\subsection{The geometry of subgroups of direct products of free groups}

Subgroups of direct products of free groups form a large and natural class of groups, which is the source of many important pathological examples in group theory \cite{Mih-68,stallings,bieri}. In particular, they provided the examples of groups of type $\mathcal{F}_{n-1}$ and not $\mathcal{F}_n$ for every $n\in \mathbb{N}$ \cite{stallings,bieri}. Here we call a group of type $\mathcal{F}_n$, if it has a classifying space with finitely many cells of dimension $\leq n$. On the other hand, finitely presented SPFs were classified in terms of their finiteness properties in \cite{BHMS-09,BHMS-13}. We refer to \cite{UniformBounds-25} for further background and details on SPFs.

	The above results make it natural to study the finer geometric invariants of SPFs, and Dehn functions are one such invariant. In this work we continue our study of Dehn functions of SPFs, which we initiated in \cite{UniformBounds-25}, and which generalises and strengthens previous works (see e.g.~\cite{Dison-08-II,LlosaTessera,KrophollerLlosa}). While in \cite{UniformBounds-25} we focus on obtaining general uniform upper bounds on natural classes of SPFs, here we focus on precise computations of their Dehn functions.

	Precise computations of Dehn functions are a challenging task, since they require attaining matching upper and lower bounds. The challenge in proving optimal upper bounds is that they usually require us to explicitly find the most efficient way of simplifying words using the relations of the group under consideration. While this can be difficult, proving sharp lower bounds usually turns out to be even harder, since it requires finding invariants that can be used to prove the optimality of the fillings of a suitable family of words. Our situation seems to be no exception in this respect.

\subsection{SPFs with quadratic Dehn function}

The only situation in which we can avoid finding lower bounds is when we can prove that the Dehn function is bounded above by $N^2$. This is because every SPF that is not free contains $\mathbb{Z}^2$ and is thus not hyperbolic, meaning that its Dehn function is $\succcurlyeq N^2$. Generalising the computation of the Dehn function of Stallings--Bieri groups by Carter and Forester \cite{carter2017stallings}, Kropholler and the fourth author proved for large classes of SPFs that they have quadratic Dehn function.

		More precisely, given an SPF $K$ of type $\mathcal{F}_{n-1}$ in a direct product of $n\geq 4$ free groups, they prove that $\delta_K(N)\asymp N^2$ whenever $n \geq 4\cdot\left\lceil \frac{r}{2}\right\rceil$, where $r$ is the corank \cite[Theorem 1.5]{KrophollerLlosa}. \cref{main:quadratic-case} generalises their result to all cases where $n\geq r+2$.

		To prove \cref{main:quadratic-case}, we define a normal form for elements of $K$. It consists of several pieces: when interpreted inside the ambient product of free groups, each of these pieces is supported on only $2$ of the factors, except for the last ``diagonal'' piece, which is supported on $r+1$ factors. We then perform algebraic manipulations, using the presentations found in \cite{UniformBounds-25}, to bring the product of two normal forms again into normal form. In order to perform these manipulations efficiently, we need one extra variable to play with, explaining the condition $n\ge r+2$.

Interestingly, if one thinks about Dehn functions of coabelian SPFs in terms of the geometric approach pursued by Carter and Forester in \cite{carter2017stallings}, then it amounts to pushing a filling for a loop in a direct product of $n$ trees into the level set under a $\psi$-equivariant height map $T_{m_1}\times \cdots \times T_{m_n}\to \mathbb{R}^r$. These level sets are simply connected
precisely when $n\geq r+2$. This raises the question if \cref{main:quadratic-case} is optimal.
\begin{question}
	If $K$ is an SPF which is virtually coabelian of corank $r$ in a direct product of $n \leq r + 1$ non-abelian free groups, is then necessarily $\delta_K(N)\succ N^2$?
\end{question}

\subsection{A SPF with quartic Dehn function}

This brings us back to the more challenging problem of determining the precise Dehn functions of SPFs in cases when it is not quadratic. A natural first candidate for solving this problem is the Bridson--Dison group $K_2^3(2)$, as it is the first and simplest example of an SPF for which we know that its Dehn function is not quadratic. Dison \cite{Dison-08,Dison-09} and Bridson \cite{BridsonPersonal} proved that $N^3\preccurlyeq \delta_{K_2^3(2)}(N)\preccurlyeq N^6$. This already shows that \cref{main:quadratic-case} cannot hold in general if we just require $n\geq r+1$.
\Cref{main:disons-group} shows that neither the upper nor the lower bound are optimal and that in fact $K_2^3(2)$ has quartic Dehn function.

Both the upper and lower bounds in \cref{main:disons-group} involve new ideas compared to the previous literature on SPFs,
which may be of independent interest. For the upper bound, instead of pushing with respect to the surjective homomorphism $F_2\times F_2\times F_2\to \mathbb{Z}^2$ with kernel $K_2^3(2)$, our strategy is to  interpret $K_2^3(2)$ as a kernel of a morphism from the Stallings--Bieri group $\SB[3]$ to  $\mathbb{Z}$ and apply the pushing argument to this group, using that it admits $(N^2,N)$ as an area-radius pair.

The main innovation in our proof of \cref{main:disons-group} is however the lower bound, which relies on an obstruction that to us seems to be completely novel. Our strategy here is to introduce a new invariant on the set of null-homotopic words $\trivialwords$ in a suitable generating set of $K_2^3(2)$, which we call the \emph{braid-invariant}. The braid-invariant induces lower area bounds on null-homotopic words and we apply it to show that there is a family of such words with quartic area growth in terms of their word length.
To define the braid-invariant, we consider the pure braid group $ \PB_3 $ on three strands, interpreted as the fundamental group of the ordered configuration space of three points in the plane; then we construct a homomorphism from $\trivialwords$ to $\PB_3$ for every suitable choice of base points on the plane. Roughly speaking this map is defined by a certain kind of braiding of words described by a well-chosen generating set $\left\{x_1,x_2,y_1,y_2\right\}$ of $K_2^3(2)$; here the set $\left\{x_1,x_2\right\}$ (resp.~$\left\{y_1,y_2\right\}$) diagonally generates the kernel of the restriction of $\psi$ to the first and third factor (resp.~second and third factor). The braiding happens between the words in the first pair of generators and the words in the second pair of generators. The braid-invariant counts the number of choices of base points that produce a nontrivial braid, and we prove that its growth can be quartic for a suitably chosen family of words.

We conclude by observing that \cref{main:disons-group} gives a negative answer to
\cite[Question 3]{LlosaTessera}, which asks if the Dehn function of every finitely presented SPF in a direct product of $n$ free groups is bounded above by $N^n$.

\subsection*{Structure}
The paper is structured as follows.%
\begin{itemize}
	\item In \cref{sec:notation-p} we fix the definitions and the notation. We recall the presentations for the kernels, as well as the details of the push-down strategy, from our previous paper \cite{UniformBounds-25}.

	\item In \cref{sec:quadratic} we prove \cref{main:quadratic-case}, by defining a normal form and performing some word manipulations. This does not rely on the push-down argument.
	\item In \cref{sec:disons-group} we prove \cref{main:disons-group}. For the upper bound, we use the push-down technique using the Stallings-Bieri group as ambient group. For the lower bound, we introduce and use the braid-invariant.

\end{itemize}

\subsection*{Guide for the reader}

\Cref{sec:notation-p} is required for understanding the rest of the paper.
\cref{sec:quadratic,sec:disons-group} are fairly independent to one another, so the reader interested in only one of the main theorems can jump directly from \cref{sec:notation-p} to the relevant section.

\subsection*{Acknowledgements} The first author was supported by the Basque Government grant IT1483-22.
The second author would like to thank Karlsruhe Institute of Technology for the hospitality, and was partially supported by the INdAM GNSAGA Project, CUP E55F22000270001. The third author gratefully acknowledges support from the Royal Society through the Newton International Fellowship (award number: NIF\textbackslash R1\textbackslash 231857). The fourth author would like to thank Robert Kropholler and Romain Tessera for many discussions about the topics of this work. The fourth and the fifth author gratefully acknowledge funding by the DFG 281869850 (RTG 2229). %
For the purpose of Open Access, the authors have applied a CC-BY public copyright licence to any Author Accepted Manuscript (AAM) version arising from this submission.

\addtocontents{toc}{\protect\setcounter{tocdepth}{2}}

\def\r{}

\section{Preliminaries and notation}\label{sec:notation-p}
We start by recalling some preliminary notions and introducing notation from \cite[Section 2]{UniformBounds-25}. In doing so, we stay close to the contents and notation introduced there.

For a group $G$, we denote by $\inv g = g^{-1}$ the inverse of an element $g\in G$, and we denote by $ [g,h] = gh \inv g \inv h$ the commutator of two elements $g,h\in G$.

\subsection{Free groups and homomorphisms}\label{sec:substitutions}
We denote by $F(S)$ the free group with basis a set $S$. Elements $w\in F(S)$ can be written uniquely as reduced words in the alphabet $S\sqcup S^{-1}$; the group operation corresponds to concatenation of words, followed by cancellation of adjacent pairs of inverse letters. We define the \emph{length} $\abs{w}_S$ as the number of letters in the word representing $w$.

Let $S=(s_1,\dots,s_n)$ be a finite ordered tuple, and let $w=w(s_1,\dots ,s_n)\in F(S)$ be an element. Given another set $T$, and given $u_1,\dots,u_n\in F(T)$, we denote by $w(u_1,\dots,u_n)\in F(T)$ the word obtained from $w$ by substituting each occurrence of $s_i$ with $u_i$, for $i \in \range n$. More precisely, we consider the unique homomorphism $\eta \colon F(S)\rar F(T)$ satisfying $\eta(s_i)=u_i$ for $i \in \range n$, and we define $w(u_1,\dots,u_n)\coloneq\eta(w)$.

\subsection{Area in free groups}\label{sec:definition-area}
Let $F$ be a free group and let $\cR\subseteq F$ be a subset. For $w\in F$ define the \textit{area} of $w$ as
\[
	\Area_{\mathcal R}(w) = \inf \set*{M\in\bN \Suchthat w = \prod_{i=1}^M \inv u_i R_i u_i,\,u_i \in F,\,R_i \in \mathcal R} ,
\]
and set $\Area_\cR(w)=+\infty$ if $w$ does not belong to the normal subgroup generated by $\cR$.

\subsection{The free group over a subset of a group}
Given a group $G$ and an arbitrary subset $S\subseteq G$, there is a unique homomorphism 
\[F(S)\rar G\]
sending each element of $S$ to the corresponding element of $G$. This map is surjective if and only if $S$ is a generating set for $G$. For two elements $u,w\in F(S)$, when we write $u=w$ we mean this as an equality in the free group $F(S)$, whereas when we write $u \eqg{G} w$ we mean that $u$ and $w$ project to the same element of $G$.

For $ g \in G $, we denote by $ \abs{g}_S = \inf\{\abs{w}_S : w\in F(S), \ w$ projects to $g\}$. We set $\abs{g}_S=+\infty$ if $g$ does not belong to the subgroup generated by $S$. If $S$ is a finite generating set for the group $G$, then $\abs{g}_S$ is the distance from the origin to $g$ in the Cayley graph $\Cay{G}{S}$.

\subsection{Dehn function of a group}
Let $G = \presentation S \relations$ be a finitely presented group, where $S\subseteq G$ is a finite set of generators and $\cR\subseteq F(S)$ is a finite set. Define the \emph{Dehn function} $ \delta_{G,S,\cR} \colon \NN \to \NN$ as follows:
\[
	\delta_{G,S,\cR} (N) = \max\set{\Area_{\relations}(w) \suchthat w \in F(S), \abs{w}_S \leq N, w=_G1}.
\]

For two functions $ f,g \colon \NN \to \NN$, we denote $ f \asympleq g $ if $ f(N) \leq C g(CN+C) + CN + C $ for some constant $C>0$. We denote $ f \asymp g $ if $ f \asympleq g $ and $ f \asympgeq g $, and we observe that $\asymp$ is an equivalence relation.

Given two different finite presentations $G = \presentation S \cR \cong \presentation {S'} {\cR'}$ for the same group $G$, we have that $\delta_{G,S,\cR} \asymp \delta_{G,S',\cR'}$. We define the \textit{Dehn function} $\delta_G$ as the equivalence class of the functions $\delta_{G,S,\cR}$, associated to all the possible presentations for $G$, up to the equivalence relation $\asymp$. With an abuse of notation, we will sometimes denote by $\delta_G:\bN\rar\bN$ a function in the equivalence class.

\subsection{Area-radius pairs for a group}
Let $G=\ppres{S}{\cR}$ be a finitely presented group. An \emph{area-radius pair} for the presentation is a pair $(\alpha,\rho)$ of functions $\alpha,\rho \colon \bN\rar\bN$ satisfying the following property: for every $w\in F(S)$ such that $ w \eqg{G} 1 $, it is possible to write $w=\prod_{i=1}^k \ol{u}_iR_iu_i$ for some $u_1, \dots ,u_k\in F(S)$ and $R_1, \dots ,R_k\in\cR$, such that $k\le \alpha(\abs{w})$ and $\abs{u_1}, \dots ,\abs{u_k}\le \rho(\abs{w})$. While the Dehn function only controls the number of relations used to fill a word $w$, an area-radius pair bounds at the same time also the length of the conjugating elements $u_i$ in the identity $w=\prod_{i=1}^k \ol{u}_iR_iu_i$.

We observe that, if $G=\ppres{S}{\cR}\cong\ppres{S'}{\cR'}$ are two different presentations for the same group, and if $(\alpha,\rho)$ is an area-radius pair for $G$ with respect to the first presentation, then there is an area-radius pair $(\alpha',\rho')$ for $G$ with respect to the second presentation satisfying $\alpha'\asymp\alpha$ and $\rho'\asymp\rho$. We also note that, if $(\alpha,\rho)$ is an area-radius pair for $G$, then we have $\alpha\asympgeq\delta_G$.

\begin{proposition}[Papasoglu \cite{papasoglu1996asymptotic}]\label{prop:linear-radius}
	Let $G$ be a finitely presented group and suppose that $\delta_G(N)\asymp N^2$. Then there is an area-radius pair $(\alpha_G,\rho_G)$ for $G$ with $\alpha_G(N)\asymp N^2$ and $\rho_G(N)\asymp N$.
\end{proposition}
\begin{proof}
	This is proved in \cite{papasoglu1996asymptotic} on page 799.
\end{proof}

\subsection{Push-down map}\label{sec:push-down}
\renewcommand{\fibration}{h}
We now recall the statement of the push-down technique as described in \cite{UniformBounds-25} for the short exact sequence
\[
	\shortexactsequence K[\iota]G[\fibration]{\mathbb Z}[,]
\]
where $G = \presentation\fpgens\fprels$ is finitely presented, and $K = \generatedby\cX$ is finitely generated. 

Choose a lift $ \tilde\iota \colon  \free\cX \to \free\fpgens $ making the following diagram commute:
\[
	\begin{tikzcd}
		& \free\cX \arrow[d] \arrow[r, "\tilde\iota"] & \free\fpgens \arrow[d]\arrow[dr, "\widetilde \fibration"] \\
		1 \arrow[r] & K \arrow[r, "\iota"] & G \arrow[r, "\fibration"] & \mathbb Z \arrow[r] & 1,
	\end{tikzcd}
\]
where $ \widetilde \fibration \colon \free \fpgens \to \mathbb Z$ is defined by composition.

\begin{definition}\label{def:push-down}
	A \emph{push-down map} is a map (not necessarily a group homomorphism)
	\[
		\bigfun{\push}{\mathbb Z \times \free\fpgens}{\free\cX}{(q, w)}{\push_q(w)}
	\]
	satisfying the following conditions:
	\begin{enumerate}
		\item For every $q\in \mathbb Z$ and $ w, w' \in \free\fpgens $ we have
		      \[
			      \push_q(w \cdot w') = \push_q(w) \cdot \push_{q + \widetilde\fibration(w)}(w').
		      \]
		\item For every $x\in\cX$ we have $\push_{0}(\tilde\iota(x))=_Kx$.
	\end{enumerate}
\end{definition}

\begin{lemma}\label{lem:push-down-existence}
	For every $q\in \mathbb Z$, choose $u_q\in F(\fpgens)$ with $\widetilde\fibration(u_q)=q$; set $u_{0}=1$. For every $q\in \mathbb Z$ and $a\in\fpgens$, choose an element $z_{q,a}\in F(\cX)$ such that $\tilde\iota(z_{q,a})=_G u_q\cdot a\cdot \ol{u}_{q+\widetilde\fibration(a)}$. Then there is a unique push-down map $\push\colon \mathbb Z\times F(\fpgens) \rar F(\cX)$ satisfying $\push_q(a)=z_{q,a}$.
\end{lemma}
\begin{proof}
    This is \cite[Lemma 2.10]{UniformBounds-25} for $Q= \mathbb Z$.
\end{proof}

\begin{theorem}[Push-down]\label{thm:push-down}
	Suppose that we are given a short exact sequence
	\[
		\shortexactsequence KG{\mathbb Z}
	\]
	with $ G = \presentation{\fpgens}{\mathcal C} $ and $K = \presentation{\cX}{\mathcal R}$ finitely presented. Let $ (\alpha_G, \rho_G) $ be an area-radius pair for $G$. Let $ \push $ be a push-down map for the sequence. Suppose that
	\[
		\max_{C \in \fprels, \abs{q}_{\widetilde\fibration(\fpgens )} \leq N}\Area_\relations(\push_q(C)) \leq f(N)
	\]
	for some function $ f \colon \NN \to \NN $. Then the Dehn function of $K$ satisfies \[
		\Dehn{K} \asympleq \alpha_{G}(N) \cdot f(\radius{G}).
	\]
\end{theorem}

\begin{proof}
    See \cite[Theorem 2.11]{UniformBounds-25} for $Q=\mathbb Z$.
\end{proof}

\renewcommand{\fibration}{\psi}

\def\r{}
\def\n{}
\RenewDocumentCommand{\K}{O{n} O{r} O{}}{K^{#1}_{#2}(\ifblank{#3}{#2}{#3})}
\RenewDocumentCommand{\Fii}{m O{r}}{F^{(#1)}_{#2}}
\RenewDocumentCommand{\dirprod}{O{r} O{n}}{%
	\ifthenelse{\equal{#2}{3}}{%
		\freecopy{#1}{a} \times \freecopy{#1}{b} \times \freecopy{#1}{c}%
	}{%
		\freecopy{#1}{1} \times \dots \times \freecopy{#1}{#2}%
	}%
}

\section{Quadratic Dehn function}\label{sec:quadratic}

	We are now interested in the Dehn functions of subgroups of type $\mathcal F_{n-1}$ in a direct product of $n$ free groups. As we have seen in the previous paper \cite[Section 4.3]{UniformBounds-25}, all these subgroups are commensurable to either direct products of (finitely many, finitely generated) free groups, or to one of the groups $K_{m_1,\dots,m_n}(r)$ for integers $m_1,\dots,m_n,~r$ with $m_1,\dots,m_n\geq r\geq 1$. Here the group $K_{m_1,\dots,m_n}(r)$ is defined as the kernel of a homomorphisms $\psi\colon F_{m_1}\times \cdots \times F_{m_n}\to \mathbb{Z}^r$ from a direct product of $n$ free groups of ranks $m_1,\dots,m_n$ to a free abelian group of rank $r$, where the homomorphism $\psi$ is required to be surjective on each factor. Note that different choices of $\psi$ give isomorphic groups (see for example \cite{Dison-08}).

    The Dehn function of $K_{m_1,\dots,m_n}(r)$ is independent of the values of $m_1,\dots,m_n\geq2$, and it only depends on $n$ and $r$ (see \cite[Theorem 4.4]{UniformBounds-25}). Moreover, for $r=1$ the Dehn function had already been proved to be quadratic for all $n\geq 3$ \cite{carter2017stallings}. Therefore, in the following we will only focus on the groups $\K := K_{r,\dots,r}(r)$ for $r\geq 2$, where $n$ denotes the number of factors.

In the previous paper \cite{UniformBounds-25}, we have shown that the Dehn functions of $\K$ have a uniform polynomial upper bound, independent of $n$ or $r$.  This was done using the push-down technique of \cref{thm:push-down}. The upper bound we produce is, however, always at least quartic: even if in some cases the area of the push of the relations is quadratic, by \cref{thm:push-down} the area of the push must be multiplied by the Dehn function of the ambient group to get an estimate for the Dehn function of the kernel. This estimate is indeed not sharp in many cases. Kropholler and the fourth author had proved that, whenever $ \ceil{\frac r2} \leq \frac n4 $, the Dehn function of $ \K $ is quadratic \cite{KrophollerLlosa}.

In this section, we aim to improve this result, and prove that $ \Dehn\K \asymp N^2 $ whenever $ n \geq r+2 \ge 4$. The proof strategy goes as follows: for every element $g\in \K$, we define a \emph{normal form} associated with it, which is a word representing $g$ canonically.
By using this normal form, it is possible to subdivide the van Kampen diagram associated with a trivial word $w$ into triangles similarly to what was done in \cite{carter2017stallings,KrophollerLlosa}. Then, by manipulating the normal form, we are able to prove that every such triangle may be filled in so that its area is bounded by a quadratic function of its perimeter. It will then follow that the total area of the van Kampen diagram is bounded by a quadratic function in the length of the word $w$.

\subsection{Presentations for \texorpdfstring{$\K$}{Kⁿᵣ(r)}}\label{sec:presentations}

	Fix integers $n\geq r+2\geq 2$.
	For $ \alpha \in \range n $, let
	\[\freecopy r \alpha = \left\langle\fpgen 1 \alpha,\dots, \fpgen r \alpha\right\rangle;\]
    be a non-abelian free group of rank $r$. Denote by $ \fpgens[\alpha] $ the ordered tuple of generators
	\[\Aii \alpha =\left(\aii \alpha1,\ldots,\aii \alpha r\right)\]
	and define
	\[\A=\bigcup_{\alpha=1}^n \Aii \alpha.\]

	We define
	\[
		\K=\ker\left(\fibration \colon \dirprod \to \ZZ^r\right),
	\]
	where $\fibration$ is the surjective morphism sending $\fpgen j \alpha$ to the $j$-th basis vector $e_j$. We now recall from \cite{UniformBounds-25} explicit finite presentations for these groups, which we are going to need in the next sections.

	We denote by
	\[\prii \alpha \colon \dirprod \to \Fii \alpha\]
	the projection onto the $\alpha$-th factor. We also denote by $ \prii \alpha $ the restriction $\prii \alpha\colon \K \to \Fii \alpha$ to the kernel.
	Moreover, in the following we often regard $\Fii \alpha $ as a subgroup of $\freecopy r1 \times \dots \times \freecopy rn$, with respect to the natural inclusion.

	Define 
	\[\xii \alpha i = \fpgen i\alpha  \invfpgen in \in \K\]
	for $ \alpha \in \range {n-1} $ and $ i \in \range r $. Define $ \Xii \alpha $ as ordered tuple
	\[
		\Xii \alpha=\left(\xii \alpha 1,\dots,\xii \alpha r\right).
	\]
	and set
	\[\X=\bigcup_{\alpha = 1}^{n-1}\Xii \alpha.\]
    From \cite[Lemma 3.1]{UniformBounds-25} we have that the set $\X$ generates $\K$.

	Consider the following sets of relations:
	\begin{align*}
		\R_1 & \coloneq \set*{\left[x_i, y_i \right] \Suchthat
		{\begin{aligned}
					  & x_i=\kgen i\alpha ,\ y_i=\kgen i\beta               \\
					  & i \in \range r,\ \alpha \neq \beta \in \range {n-1}
				 \end{aligned}}}        \\
		\R_2 & \coloneq \set*{\left[x_i, y_j\ol{z}_j \right] \Suchthat
			\begin{aligned}
				 & x_i=\kgen i\alpha ,\ y_i=\kgen i\beta ,\ z_i=\kgen i\gamma        \\
				 & i \neq j \in \range r,                                            \\
				 & \alpha, \beta, \gamma  \in \range {n-1} \text{ pairwise distinct}
			\end{aligned}
		}
	\end{align*}
	and define $\mathcal R := \R_1 \cup \R_2$. By \cite[Theorem 3.13]{UniformBounds-25}, a presentation for $\K$ is given by $\presentation \X \R$ for all $n\geq r+2\geq 2$. Note that the presentations are more complicated if we remove the conditions $n\geq r+2\geq 2$, see \cite[Section 3.1.2]{UniformBounds-25}.

\subsection{Standard and diagonal free subgroups of \texorpdfstring{$\K$}{Kⁿᵣ(r)}}\label{sec:free-subgroups}

In what follows,
we will often need to consider the same word written in different alphabets as described in \cref{sec:substitutions}.
Recall that, if $w\in\free{\afgen 1,\ldots, \afgen t}$, then $w(\eta_1,\ldots, \eta_t)$ is the element in $\free{\eta_1,\ldots,\eta_t}$ obtained from $w$ by substituting $\afgen i$ with $\eta_i$.

In addition to the standard subsets of generators $ \kgens[\alpha] \subseteq \kgens$ for $\K$, we often employ a diagonal set $ \diaggens = (\kgen 11, \dots, \kgen rr) $. This is the key point where we use that $ n \geq r + 2 $: this assumption implies that $ \Delta $ is well-defined, and moreover we have some leeway from the fact that $ \kgens[r+1] $ is also well-defined (i.e.~$ r+1 \leq n-1 $) and is disjoint from $ \Delta $. This will be crucial in order to perform some word manipulations with a small number of relations.
Similarly, we define $ \inv \Delta \coloneq (\invkgen 11, \dots, \invkgen rr) $.

\begin{lemma}\label{lem:alphabet-X-generate-free-group}
	The subgroup of $\K$ represented by words written in the alphabet $\kgens[i]$ is a free group for $ i \in \range {n-1}$.
	The subgroup of $\K$ represented by words written in the alphabet $\Delta$ is a free group.
\end{lemma}
\begin{proof}
	Call $X^{(i)} $ (resp.~$D$) the subgroup of $\K$ represented by words written in the alphabet $\kgens[i]$ (resp.~$ \Delta$).
	The map
	$X^{(i)}\to \langle \fpgens[i]\rangle$
	(resp.~$ D\to\langle  \fpgens[n]\rangle $) sending $x\toi_j$
	to $a\toi_j$ (resp.~$x\toi_i$
	to $a\ton_j$) induces an injective homomorphism of groups. Now, the statement follows from the fact that $\langle \fpgens[i]\rangle$ is a free group for all $i \in \range n$.
\end{proof}

\begin{remark}
	In the lemma above, we proved that for some subsets $ S $ of the generators $ \kgens $ of $\K$, the subgroup $\generatedby S < \K$ is a free group with $S$ as generating set.
	However, we should keep in mind the distinction between $ \generatedby S < \K $, and the \emph{abstract} free group $ \free S $ with generating set $S$, even if they are naturally isomorphic: the former is a subgroup of $\K$, while the latter is naturally a subgroup of $ \free{\kgens} $. In other words, elements of $ \generatedby S $ are elements of $\K$, while elements of $ \free S $ are formal words, up to free equivalence, in the alphabet $ S \subseteq \kgens $.

	For the same reason, we also keep two different symbols to denote $ \freecopy ri \isom \free{\fpgens[i]}$: the first should be thought of as a subgroup of $ \dirprod $, while the second denotes words, up to free equivalence, in the letters $ \fpgen ji $. This is particularly useful as we will be able to perform substitutions as described in \cref{sec:substitutions}.
\end{remark}
\begin{definition}
	Let $ w \in \free{S} $ be an element of a free group with generating set $S$, and let $ s \in S $. We say that $w$ is \emph{$s$-balanced} if $w$ is in the kernel of the homomorphism $ \free{S} \to \ZZ $ that sends $ s \mapsto 1 $, and $ s' \mapsto 0 $ for all $ s' \in S \setminus \set s $, or equivalently, if any word in the alphabet $S \cup S^{-1}$ that represents $w$ contains an equal number of $s$ and $ \inv s $.
\end{definition}

\subsection{Normal form}
In this section, we associate to each $g\in K$ a canonical element of $ \free\kgens $  representing it.

\begin{definition}\label{def:normal-form}
	Let $g \in \K$. The  \emph{normal form} of $g$ is an element $ \normal g \in \free\kgens $ representing $g$ that decomposes as
	\[
		\normal g = \normal[d]g \cdot \normal[1]g \cdot \normal[2]g \cdots \normal[n-1]g,
	\]
	where $\normal[i]g \in \free{\kgens[i]}$ for $ i \in \range {n-1} $ and $\normal[d]g$ belongs to the commutator subgroup $\commsub{\free{\Delta}}$.
\end{definition}

\begin{lemma}
	The element $ \normal[d]g $ represents an element in the subgroup
	\[
		\commsub*{\freecopy rn} \subset \K \subset \dirprod.
	\]
\end{lemma}

\begin{proof}
	For $ i \in \range{n-1} $, the projection $ \prii{i}: \dirprod \to \freecopy ri $ sends $ \kgen jj $ to the trivial element for $ i \neq j $, and to $ \invfpgen ii $ for $ i = j $. Since $ \normal[d]g $ belongs to the commutator subgroup, it is $ \kgen ii $-balanced, so $ \prii{i}(\normal[d]g) $ is trivial.

	Since the projection to every factor except the last is trivial, then $ \normal[d]g $ represents an element in $ \freecopy rn $. However, since it is also an element of $\K$, it is contained in $ \freecopy rn \cap \K = \commsub{\freecopy rn} $.
\end{proof}

The following lemma assures that the normal form always exists and is unique.

\begin{lemma}\label{lem:norml-form-exists-unique}
	For every $g$ in $\K$, there exists a unique normal form. Moreover, for $ 1 \leq i \leq n-1 $, the factor $ \normal[i]g \in \free{\kgens[i]} $ is such that $ \normal[i]g(\fpgens[i]) $ represents the projection $ \prii{i} (g) $.%
\end{lemma}

\begin{proof}
	We will first prove existence and then uniqueness.

	\emph{Existence.} For every $i \in \range n$, denote by $ g_i \in \freecopy ri = \free{\fpgens[i]} $ the projection of $g$ to the $i$-th factor. For $ i \in \range {n -1}$, we set $ \normal[i] g = g_i(\kgens[i]) \in \free{\kgens[i]} $: we have in particular that $ \prii{i} ( \normal[i]g ) = g_i $. Moreover, we define
	\[
		\normal[d]{g} \coloneq g_n\big(\inv\diaggens\big) \cdot \left(g_1\big(\diaggens\big) \cdots g_{n-1}\big(\diaggens\big)\right)^{-1}.
	\]

	We claim that the element $\normal[d]{g}$ is in $ \commsub{\free\diaggens} $.
	Indeed, if we project $ \normal[d]{g} $ to the $i$-th factor, for $ i \in \range r $, we get
	\[
		\left(\fpgen ii\right)^{\fibration_i(g_n)} \cdot \left(\left(\fpgen ii\right)^{-\fibration_i(g_1)} \cdots \left(\fpgen ii\right)^{-\fibration_i(g_{n-1})}\right)^{-1}
	\]
	where $ \fibration_i $ is the $i$-th component of the map $ \fibration \colon \dirprod \to \ZZ^r $, while for $i \in \range [r +1] {n -1}$ the projection is trivial.  Since $g$ is represented by $ g_1 \cdots g_n $ and $ \fibration(g) = 0 $ as $g \in \K$, the above expression is trivial.
	So $ \normal[d]{g} $ is $ \kgen ii $-balanced for all $i$, i.e.~it belongs to the commutator subgroup $ \commsub{\free\diaggens} $.

	Thus, the word
	\[
		w_g = \normal[d]{g} \cdot \normal[1]{g} \cdot \normal[2]{g} \cdots \normal[n-1]{g} \in \free{\kgens}
	\]
	satisfies the properties of the normal form. We just need to show that $w_g$ represents $g$.
	It is enough to prove that the projection $\prii i(w_g)$ coincides with $ g_i $, for all $ i \in \range  n$. For $ i \in \range{n-1} $, we observe that the projection $\prii j\big(\normal[i]{g}\big)$ of $w_g\toi$ satisfies:
	\[
		\prii j\big(\normal[i]{g}\big) =
		\begin{cases}
			1                             & \text{if } j \not\in \set{i, n}; \\
			g_i \left(\fpgens[i]\right)   & \text{if } j = i;                \\
			g_i\left(\invfpgens[n]\right) & \text{if } j=n.
		\end{cases}
	\]
	On the other hand, $\prii j\big(\normal[d]{g}\big)$ is nontrivial only if $j=n$ and in this case it is represented by the word
	\[
		\normal[d]{g}\left(\invfpgens[n]\right) = g_n\left(\fpgens[n]\right) \cdot \left(g_1\left(\invfpgens[n]\right) \cdots g_{n-1}\left(\invfpgens[n]\right)\right)^{-1}.
	\]
	Putting everything together, we obtain that
	\[
		\prii{i}(w_g) = g_i\left(\fpgens[i]\right),
	\]
	for every $i\in \range n$, as desired.

	\emph{Uniqueness.} Let
	\[
		v_g = v_g^\Delta \cdot v_g^{(1)}\cdot v_g^{(2)}\cdots v_g^{(n-1)}
	\]
	be another normal form for $g$ and let $w_g$ be as in the proof of the existence.%
	Since $v_g^\Delta$ represents an element in $F\ton_r$, its projection $\prii i(v_g^\Delta)$ represents the trivial element whenever $ i \in \range{ n-1}$.
	It is also clear from the definition that $\prii i(v^{(j)}_g)$ is nontrivial only if $j=i$ or $j=n$. Therefore, for $i \in \range{ n-1}$,
	\[\prii{i}(v_g^{(i)}) = \prii{i} (v_g) = \prii{i}(w_g) = \prii{i}(\normal[i]g),\] so $ v_g^{(i)}(\fpgens[i]) = \normal[i]g(\fpgens[i]) $, which implies $ v_g^{(i)} = \normal[i]g$.

	As the two elements of $ \free\Delta $
	\[ v_g^\Delta= v_g\cdot \big(v_g^{(1)}\cdot v_g^{(2)}\cdots v_g^{(n-1)}\big)^{-1},\]
	\[ w_g^\Delta= w_g\cdot \big(w_g^{(1)}\cdot w_g^{(2)}\cdots w_g^{(n-1)}\big)^{-1}\]
	represent the same element in $\K$ (because of the previous computations) and the subgroup generated by $\Delta$ is free (\cref{lem:alphabet-X-generate-free-group}), they coincide. We have therefore proven that the normal form is unique.

	Finally, the fact that $ \normal[i]g(\fpgens[i]) $ represents the projection $ \prii{i} (g)$ follows from the construction shown in the existence part.%
\end{proof}

\begin{lemma}
	Let $g\in \K$ and $\normal g \in \free\kgens$ be its normal form; then \( \length{\normal g}\leq 3\cdot \abs g_{\kgens} \).%
\end{lemma}

\begin{proof}
	Let  $\normal g$ be the normal form of $g$ expressed as in \cref{def:normal-form}.
	Let $w \in \free \kgens$ be any word representing $g$. By \cref{lem:norml-form-exists-unique},
	\[
		\length{w_g\toi}=\length{w\toi_g\big(\fpgens[i]\big)} = \abs{\prii i(g)}_{\fpgens[i]},
	\]
	where $w\toi_g\big(\fpgens[i]\big)$ represents $\prii i(g)$ (for $1\leq i\leq g$). However, $\prii i(g)$ is represented by $\prii i(w)$, and this projection is obtained from $w$ by replacing the letters $x\toi_j$ with $a\toi_j$ and $x^{(k)}_j$ with $1$, for all $j,k \in \range n$ and $k\neq i$.
	Thus, $\length{w_g\toi}$ is at most the number of letters in $w$ belonging to the alphabet $\Xii i$.
	This is enough to prove that
	\begin{equation}\label{eq:inequality-between-length}
		\length{w_g^{(1)}\cdots w_g^{(n-1)}}\leq \length{w}.
	\end{equation}

	For the same reason, we have that $ \prii n(g) $ is obtained by replacing $ \kgen ji $ with $ \invfpgen jn$, so we get that $ \abs{\prii n(g)}_{\fpgens[n]} \leq \abs{w} $.	On the other hand, $\length{w_g^\Delta}=\length{\prii n\big(w_g^\Delta\big)}$ since, by \cref{lem:alphabet-X-generate-free-group}, the words written in the alphabet $\Delta$ generate $K$ as a free group. Therefore, we have
	\begin{align*}
		\length{w_g^\Delta}
		 & =     \length{\prii n\left(w_g^\Delta\right)}                                                        \\
		 & \leq  \length{\prii n\left(w_g\cdot\left(w_g^{(1)}\cdots w_g^{(n-1)}\right)^{-1}\right)}             \\
		 & \leq  \length{\prii n\left(w_g\right)}+\length{\prii n\left(w_g^{(1)}\cdots w_g^{(n-1)}\right)^{-1}} \\
		 & \leq  2\cdot\length{w}.
	\end{align*}
	where the last inequality is due to \cref{eq:inequality-between-length}, the fact that $\prii n(w_g)=\prii n(g)$ (as they both represent $\prii n(g)$ in the free group $\Fii n$) and the inequality $\length{\prii n(g)}\leq \length{w}$.
\end{proof}

\subsection{Quadratic upper bound} %

Throughout the whole section, we use the presentation described in \cref{sec:presentations}, and all areas will be considered with respect to it. The aim of this section is to prove the following proposition.

\begin{proposition}\label{prop:multiply-normal-forms}
	Let $ n \geq 4$ and $r \geq 2$ be natural numbers satisfying $ n \geq r+2 $. There exists $ C>0 $ such that, for all $g,h\in \K$, we have that
	\[
		\Area(w_gw_h\inv w_{gh}) \leq C \cdot \max\set*{\abs{w_g}^2, \abs{w_h}^2}.
	\]
	where $w_g,w_h,w_{gh}$ denote the normal forms of, respectively, $g, h, gh$.
\end{proposition}

In other words, this tells us that triangles in the Cayley graph whose sides represent words in normal form have quadratic area with respect to their perimeter. Given a word of length $N$, we can subdivide its van Kampen diagram into such triangles as in \cref{fig:normal-form}, so that the estimate above yields a quadratic bound on the area of the whole diagram. Similar arguments already appeared in other places, including \cite{GerSho-02, carter2017stallings}; below we give a precise statement and proof.

\begin{theorem}\label{thm:normal-form}
	Let $ n \geq 4, r \geq 2$ be natural numbers with $ n \geq r+2 $. There exists a constant $C>0$ such that for any $ g \in \K $ and for any $ w \in \free\kgens $ that represents $g$, it holds that
	\[
		\Area(\inv w w_g) \leq C \cdot \abs w^2,
	\]
	where $w_g \in \free\kgens$ denotes the normal form of $g$.
\end{theorem}

Since the normal form of the trivial element is trivial, this implies directly \cref{main:quadratic-case}. We now prove how to recover \cref{thm:normal-form} from \cref{prop:multiply-normal-forms}.   %

\begin{proof}
	Pick a constant $C>0$ such that the statement holds for all $w$ of length at most $12$. We may also assume that $C$ is bigger than the constant $C'$ given by \cref{prop:multiply-normal-forms}.

	We claim that $\Area(\inv w w_g) \leq C \cdot \abs w^2$ holds for words of any length: we prove this by induction. To this end, let $ N > 12$ and assume that the statement holds for every $w$ with $ \abs w < N $.

	Let $ w \in \free\kgens $ representing $g$ with $ \abs w = N $, and decompose it as $ w=w_1w_2 $, so that, for $i=1,2$, the word $ w_i $ has length at most $ N/2+1 $ and represents some element $ g_i \in \K $.
	By inductive hypothesis, we can replace $w_1$ and $w_2$ by $ w_{g_1} $ and $ w_{g_2} $ using at most $ 2 \cdot  C (N/2 + 1)^2 $ relations; then, we can rewrite the product $ w_{g_1} w_{g_2} $ as $ w_{g_1g_2} = w_g $ using at most $ C (N/2+1)^2 $ relations     by \cref{prop:multiply-normal-forms}.

	Thus, we are able to fill $ \inv w w_g $ using at most
	\[
		(2C+C) \cdot \left(\frac{N^2}{4} + N+1\right) < 3C \cdot \left(\frac14 + \frac1{12}\right) \cdot N^2 = C \cdot N^2
	\]
	relations, where we used that $ N+1 < N^2/12 $ for $ N>12 $. This concludes the proof.
\end{proof}

\begin{figure}
	\centering
	\begin{tikzpicture}
		\node[anchor=south west, inner sep=0] (image) at (0,0) {\rotatebox{90}{\includegraphics[height=7cm]{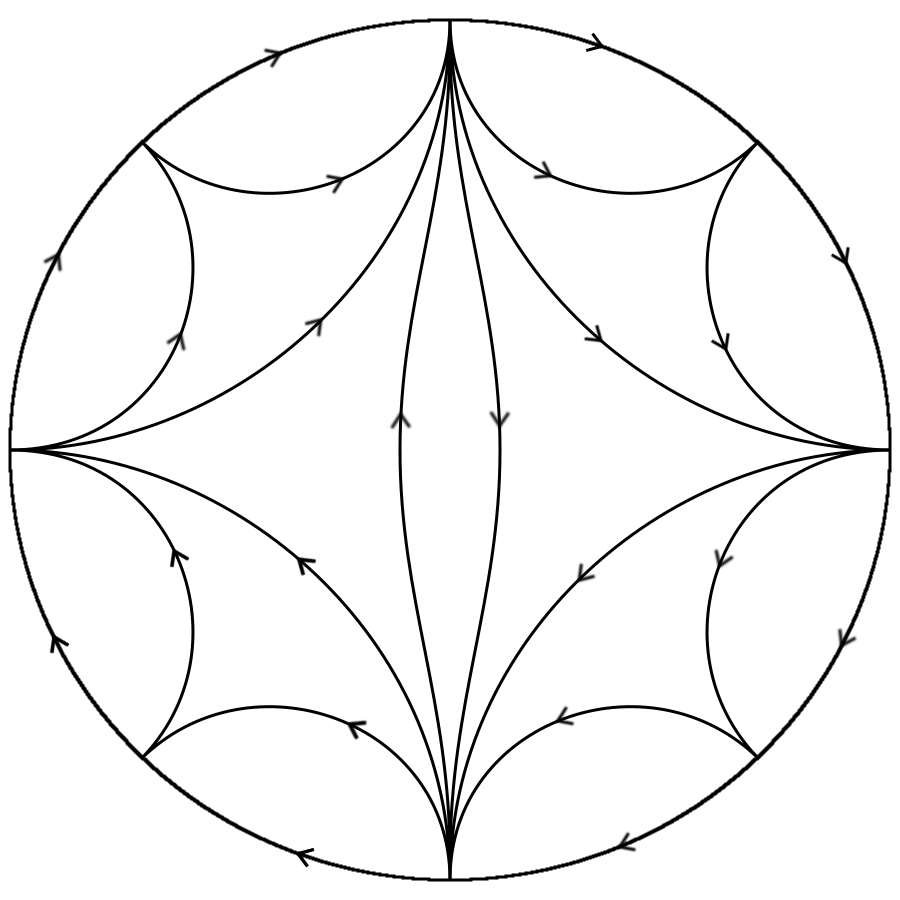}}};
		\node[above] at (3.5,3.9) {$w_{g_1g_2g_3g_4}$};
		\node[below] at (3.5,3.1) {$w_{g_5g_6g_7g_8}$};
		\node[above left] at (2.7,4.5) {$w_{g_1g_2}$};
		\node[above right] at (4.2,4.5) {$w_{g_3g_4}$};
		\node[below right] at (4.3,2.4) {$w_{g_5g_6}$};
		\node[below left] at (2.8,2.5) {$w_{g_7g_8}$};
		\node[left] at (1.5,4.5) {$w_{g_1}$};
		\node[above] at (2.3,5.5) {$w_{g_2}$};
		\node[above] at (4.7,5.5) {$w_{g_3}$};
		\node[right] at (5.5,4.5) {$w_{g_4}$};
		\node[right] at (5.5,2.5) {$w_{g_5}$};
		\node[below] at (4.7,1.4) {$w_{g_6}$};
		\node[below] at (2.3,1.4) {$w_{g_7}$};
		\node[left] at (1.5,2.5) {$w_{g_8}$};
		\node[left] at (0.5,5.0) {$g_1$};
		\node[above] at (2.0,6.6) {$g_2$};
		\node[above] at (5.0,6.6) {$g_3$};
		\node[right] at (6.5,5.0) {$g_4$};
		\node[right] at (6.5,2.0) {$g_5$};
		\node[below] at (5.0,0.4) {$g_6$};
		\node[below] at (2.0,0.4) {$g_7$};
		\node[left] at (0.5,2.0) {$g_8$};
	\end{tikzpicture}
	\caption{Given a trivial word $ w=g_1 \dots g_N $, we replace each generator $ g_i $ with its normal form $ w_{g_i} $, and then we replace iteratively the product of two normal forms with the normal form of its product. This process can be represented by the Dehn diagram above: if one proves that the area of each triangle is quadratic in its perimeter, one gets a quadratic bound of the area of the whole diagram.}
	\label{fig:normal-form}
\end{figure}
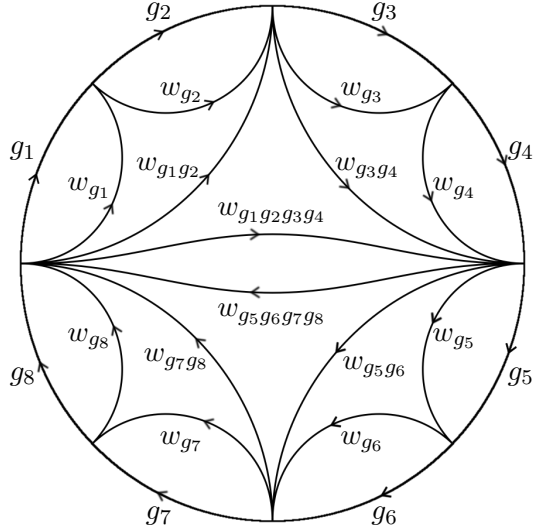

Let us delve into the proof of \cref{prop:multiply-normal-forms}. We start with some auxiliary lemmas.

\begin{lemma}\label{lemma:balanced-commute-with-diag}
	Let $ i, j \in \range r $, and $ k \in \range[r+1]{n-1} $. Then
	\[
		\left[\kgen{i}{k} \invkgen{i}{i}, \kgen{j}{j}\right] = 1.
	\]
\end{lemma}

\begin{proof}
	If $i \neq j$, then it is a relation.  Otherwise, it is immediate by checking that projections are trivial in both cases $ i=j $ and $ i \neq j $.
\end{proof}

\begin{lemma}\label{lemma:change-diagonal-small}
	Let $ i \in \range r, k \in \range[r+1]{n-1} $, and let $ \diagalph' $ be the alphabet obtained from $\diagalph$
	by replacing $ \kgen{i}{i} $ with $ \kgen{i}{k} $. There exists a constant $C>0$ such that for every $ w \in \commsub{F(\Delta)} $ we have
	\[
		\Area\big(w(\diagalph) \inv w(\diagalph')\big) \leq C \cdot \abs w^2.
	\]
\end{lemma}

\begin{proof}
	Replace in $ w(\diagalph) $ every $ \kgen{i}{i} $ with $  \left(\kgen{i}{i} \invkgen{i}{k}\right) \kgen{i}{k}$. Since $ \kgen{i}{i} \invkgen{i}{k} $ commutes with every letter of $ \diagalph $
	by \cref{lemma:balanced-commute-with-diag}, and it also commutes with $ \kgen ik $, we can commute every $ (\kgen{i}{i} \invkgen{i}{k}) $ to the left. Since $w(\diagalph)$ is $ \kgen ii $-balanced, the resulting word is freely trivial.
\end{proof}

\begin{lemma}\label{lemma:change-diagonal}
	Let $ \sigma \colon \range{r} \to \range{n-1} $ be injective, and denote by $ \diagalph' $ the alphabet $ (\kgen{1}{\sigma(1)}, \dots, \kgen{r}{\sigma(r)}) $. There exists a constant $C>0$ such that for every $ w \in \commsub{F(\Delta)} $ we have
	\[
		\Area\left(w(\diagalph) \inv w(\diagalph')\right) \leq C \cdot \abs w^2.
	\]
\end{lemma}
\begin{proof}
	It follows by iterating \cref{lemma:change-diagonal-small}.
\end{proof}

\begin{lemma}\label{lemma:convert-second-diagonal}
	There exists a constant $ C>0 $ with the following property. For every $v, w \in \free \Delta $ and $ i \in \range{n-1} $ such that either $v \in \commsub{\free \Delta} $ or $i\in \range[r+1]{n-1}$ we have
	\[
		\Area\left([w(\kalph{i}), v(\diagalph)][w(\diagalph), v(\diagalph)]^{-1}\right) \leq C \cdot \max\set{\abs w^2, \abs v^2}.
	\]
\end{lemma}

\begin{proof}
	First assume that $i\in \range[r+1]{n-1}$, say $i=n-1$. Then after simplifying, the statement is equivalent to computing the area of
	\[
		w(\kalph{n-1}) v(\diagalph) \inv w(\kalph{n-1}) w(\diagalph) \inv v(\diagalph) \inv w(\diagalph).
	\]%
	Assume that $ w(\kgens[n-1]) $ ends with the letter $ \kgen{k}{n-1} $. Note that we can replace $ \kgen{k}{n-1} v(\diagalph) \invkgen{k}{n-1} $ with $ \kgen{k}{k} v(\diagalph) \invkgen{k}{k} \eqcolon v'(\diagalph)$ using a linear number of relations, since $ \kgen{k}{n-1} \invkgen{k}{k} $ commutes with every letter of $ \diagalph $, and by iterating we conclude.%

	If $i\le r$, since $v\in \commsub{\free\Delta}$, after permuting factors and using \cref{lemma:change-diagonal} we may assume $ i = n-1 > r $, reducing to previous case. This completes the proof.
\end{proof}

Now we note a purely algebraic lemma.

\begin{lemma}\label{lemma:merge-words}
	Let $G$ be a group with a given finite presentation, and let $ s_i, t_i \in G $ for $ i \in \range r $. Let $ \mathcal S = (s_1, \dots, s_r), \mathcal T = (t_1, \dots, t_r) $, and denote by $ \mathcal{\inv T S} $ the alphabet $ (\inv t_1 s_1, \dots, \inv t_r s_r) $. Suppose that $ [t_i, \inv t_j s_j ]=1 $ for all $ i,j \in \range r $.

	There exists a constant $C>0$ such that for every $ w \in \free{\mathcal S} $ the element
	\[
		\inv w(\mathcal T) \cdot w(\mathcal S) \cdot  \inv w(\mathcal{\inv TS}) \in \free{\mathcal S \cup \mathcal T}
	\]
	represents the trivial element of $G$ and has area bounded above by $ C \cdot \abs w^2 $.
\end{lemma}

\begin{proof}
We may assume, without loss of generality, that the elements in $\mathcal S \cup \mathcal T$ are generators of the given presentation, and that the commutators $[t_i, \bar t_j s_j]$ are relations.

	Let $ j \in \range r $ and $ \epsilon \in \set{\pm 1} $ such that $ w(\mathcal S) = w'(\mathcal S) s_j^\epsilon $. Then we can rewrite the expression as
	\[
		\inv t_j^\epsilon \inv w'(\mathcal T) \cdot w'(\mathcal S) s_j^\epsilon \cdot  (\inv s_j t_j)^\epsilon  \inv w'(\mathcal{\inv TS}).
	\]
	We can then simplify the $ s_j $ (possibly after commuting it with $ t_j $, if $ \epsilon = -1 $) and commute $ t_j $ to the right with $ \abs{w'} $ relations, since $t_j$ commutes with every letter of $ w'(\mathcal{\inv TS}) $. So we get
	\[
		\inv t_j^\epsilon \left(\inv w'(\mathcal T) \cdot w'(\mathcal S) \inv w'(\mathcal{\inv TS})\right)t_j^\epsilon ,
	\]
	and by iterating the procedure we conclude.
\end{proof}

\begin{lemma}\label{lemma:convert-first-diagonal}
	Let $ i \in \range {r}$ and $ j \in \range[r+1]{n-1}$. There exists a constant $C>0$ such that, for every $v, w \in \free\Delta$, we have
	\[
		\Area\left(\left[w(\kalph i), v(\kalph j)\right]\left[w(\diagalph), v(\kalph j)\right]^{-1}\right) \leq C \cdot \max\set*{\abs w^2, \abs v^2}.
	\]
\end{lemma}

\begin{proof}
	We may assume without loss of generality $ i=1$ and $ j=n-1 $. After simplifying $ v(\kalph j) $, the expression becomes a conjugate of
	\begin{equation}\label{eq:deadly-commutator}
		\left[\inv w(\diagalph) w(\kalph {1}), v(\kalph {n-1})\right]. \tag{$*$}
	\end{equation}
	Note that, for every $ k \in \ZZ $, this expression is freely equivalent to
	\[
		\left[\inv w(\diagalph) (\invkgen 11)^k (\kgen 11)^k w(\kalph {1}), v(\kalph {n-1})\right],
	\]
	so we may assume that both $ w(\diagalph) $ and $ w(\kgens[1]) $ are $ \kgen 11 $-balanced.

	Consider now the expression $\inv w(\diagalph) w(\kalph 1) $ on the left side of the commutator. By using \cref{lemma:change-diagonal}, this can be rewritten as
	\[
		\inv w\left(\kgen{1}{n-1}, \kgen 22, \kgen 33, \dots, \kgen rr\right)w\left(\kalph 1\right).
	\]
	Then, by using \cref{lemma:merge-words} with $ \mathcal T = (\kgen{1}{n-1}, \kgen 22, \kgen 33, \dots, \kgen rr)$,  $\mathcal S = \kalph 1 $, we obtain
	\[
		w\left(\invkgen{1}{n-1}\kgen11, \invkgen22\kgen21, \invkgen 33 \kgen 31, \dots, \invkgen rr\kgen r1\right).
	\]
	Since $ \kgen 22 $ commutes with $ \invkgen kk \kgen k1$,  for all $ k \in \range r $, and with $\invkgen{1}{n-1}\kgen11$ (it can be seen by checking the projection for $ k=1, k=2, k\geq 3 $),
	we may commute every occurrence of $ \kgen 22 $ to the left, obtaining
	\[
		\left(\kgen 22\right)^{-m} \cdot 	w\left(\invkgen{1}{n-1}\kgen11, \kgen21, \invkgen33 \kgen31, \dots, \invkgen rr\kgen r1\right) ,
	\]
	where $m$ is the number of times the second generator appears in $w$ (counted with sign).

	Now we apply again \cref{lemma:merge-words} with
	\begin{align*}
		\mathcal S & = \left(\invkgen 12 \kgen 11, \invkgen 22 \kgen 21, \invkgen 33 \kgen 31, \dots, \invkgen rr \kgen r1\right) \\
		\mathcal T & = \left(\invkgen 12 \kgen {1}{n-1}, \invkgen 22, 1, \dots, 1\right)
	\end{align*}
	and we get
	\[
		\left(\kgen22\right)^{-m} \cdot \inv w\left(\invkgen 12 \kgen {1}{n-1}, \invkgen 22, 1, \dots, 1\right) \cdot w\left(\invkgen 12 \kgen 11, \invkgen 22 \kgen 21, \invkgen 33 \kgen 31, \dots, \invkgen rr \kgen r1\right).
	\]

	Apply \cref{lemma:merge-words} another time to the word $ \inv w $ with $ \mathcal S=(\invkgen 12 \kgen 11, \invkgen 22, 1, \dots, 1) $ and $ \mathcal T = (\invkgen 1{n-1} \kgen 11, 1, \dots, 1) $ obtaining%
	\begin{multline*}
		\left(\kgen22\right)^{-m} \cdot  w\left(\invkgen 1{n-1} \kgen 1{1}, 1, \dots, 1\right) \cdot \inv w\left(\invkgen 12 \kgen 11, \invkgen 22, 1, \dots, 1\right)\cdot \\
		\qquad \cdot w\left(\invkgen 12 \kgen 11, \invkgen 22 \kgen 21, \invkgen 33 \kgen 31, \dots, \invkgen rr \kgen r1\right)
	\end{multline*}
	where we also used the relations $ [\kgen 11, \kgen 1{n-1} \invkgen 12]$ and $ [\kgen 12, \kgen 1{n-1}] $ at most $ \abs w $ times each.

	Now $w(\invkgen 1{n-1} \kgen 1{1}, 1, \dots, 1)$ is freely trivial because $w$ is balanced in the first generator, and since $ \kgen 23 $ commutes with $ \invkgen 12 \kgen 11 $ we can rearrange as
	\begin{equation*}
		\left(\invkgen22 \kgen 23\right)^{m} \inv w\left(\invkgen 12 \kgen 11, \invkgen 22 \kgen 23, 1, \dots, 1\right)
		w\left(\invkgen 12 \kgen 11, \invkgen 22 \kgen 21, \invkgen 33 \kgen 31, \dots, \invkgen rr \kgen r1\right).
	\end{equation*}

	All these manipulations required a number of relations that is quadratic in the length of $w$; in the end we have obtained an expression for $\inv w(\diagalph) w(\kalph 1) $ in which all the pairs of letters commute with $ \kgen k{n-1} $, for all $ k \in \range r $. So the commutator \cref{eq:deadly-commutator} can be filled with a quadratic number of relations in the length of $w$ and $v$, as required.
\end{proof}

\begin{lemma}\label{lemma:diagonal-combined}
	There is a constant $C>0$ such that for all $i,j\in \range{n-1}$ with $i\neq j$ and $u, v\in \free\Delta$ we have
	\[
		\Area\left([u(\kgens[i]),v(\kgens[j])]\cdot [u(\diaggens),v(\diaggens)]^{-1}\right)\le C \max(\abs{u}^2,\abs{v}^2).
	\]
\end{lemma}
\begin{proof}
	Up to permuting factors and using \cref{lemma:change-diagonal}, we may assume that $i\in \range[r+1]{n-1}$. By \cref{lemma:convert-first-diagonal} we can rewrite $[u(\kgens[i]),v(\kgens[j])]$ as $[u(\kgens[i]),v(\diaggens)]$. The assertion follows by applying \cref{lemma:convert-second-diagonal}.
\end{proof}

\begin{lemma}\label{lemma:multiple-swaps}
	There exists a constant $C_n$ {depending on $n$} such that, for all $N \in \NN$, and for every $ w \in \commsub{\free{\Delta}} $ with $ \abs w \leq N $ and all $ w_1, \dots, w_{n-1} \in \free{\Delta} $ with $  \abs {w_1}, \dots, \abs {w_{n-1}} \leq N $, we have that
	\[
		\Area \left(
		\left(\prod_{i=1}^{n-1} w_i(\kgens[i])\right) \cdot w(\diaggens) \cdot  \left(w'(\diaggens) \prod_{i=1}^{n-1} w_i(\kgens[i])\right)^{-1}
		\right) \leq C_{n-1} \cdot N^2,
	\]
	where $ w'(\diaggens) = \left(\prod_{i=1}^{n-1} w_i(\diaggens)\right) \cdot  w(\diaggens) \cdot \left(\prod_{i=1}^{n-1} w_i(\diaggens)\right)^{-1} $.
\end{lemma}

\begin{proof} The proof is by induction on $k\in\range[0]{n-2}$, where $w_{k+1}=\cdots =w_{n-1}=1$, the case $k=0$ being trivial. So assume there is a constant $ C_{k-1} $ such that the induction hypothesis for $ k -1$ holds and assume that $w_{k+1}=\ldots = w_{n-1}=1$.

	By \cref{lemma:convert-second-diagonal} we have that
	\[
		w_k(\kgens[k]) w(\diaggens) \eqfree \left[w_k(\kgens[k]), w(\diaggens)\right] w(\diaggens) w_k(\kgens[k]) = [w_k(\diaggens), w(\diaggens)] w(\diaggens) w_k(\kgens[k])
	\]
	using $ C \cdot N^2 $ relations for some constant $ C > 0 $. Therefore
	\begin{align*}
		\left(\prod_{i=1}^{k} w_i(\kgens[i])\right) \cdot w(\diaggens) & = 		\left(\prod_{i=1}^{k-1} w_i(\kgens[i])\right) \cdot w_k(\kgens[k]) \cdot  w(\diaggens)                                       \\
		                                                               & =\left(\prod_{i=1}^{k-1} w_i(\kgens[i])\right) \cdot  [w_k(\diaggens), w(\diaggens)] \cdot  w(\diaggens) \cdot  w_k(\kgens[k]) \\
		                                                               & =\left(\prod_{i=1}^{k-1} w_i(\kgens[i])\right) \cdot  w_k(\diaggens)  w(\diaggens) \inv w_k(\diaggens) \cdot  w_k(\kgens[k])   \\
		                                                               & = w'(\diaggens) \cdot \prod_{i=1}^{k} w_i(\kgens[i]),
	\end{align*}
	where in the last step we used the inductive hypothesis that uses at most $ C_{k-1} \cdot (3N)^2 $ relations. So by letting $ C_k = 9 C_{k-1} + C $ we conclude.
\end{proof}

\begin{proof}[Proof of \cref{prop:multiply-normal-forms}]
	We start from the product%
	\[
		\normal g \normal h = \normal[d]g(\diaggens) \cdot \left( \prod_{i=1}^{n-1} \normal[i] g(\kgens [ i ]) \right) \cdot \normal[d]h(\diaggens)\cdot \left(  \prod_{i=1}^{n-1} \normal[i]h(\kgens[i]) \right),
	\]
	and we show that we can obtain the normal form $ \normal{gh} $ of $gh$ by using a quadratic amount of relations in the lengths of $ \normal g $ and $ \normal h $.

	By applying \cref{lemma:multiple-swaps}, we can rewrite the expression as
	\[
		\normal[d]g(\diaggens) \cdot w'(\diaggens)\cdot \left( \prod_{i=1}^{n-1} \normal[i] g(\kgens [ i ]) \right)  \cdot \left(  \prod_{i=1}^{n-1} \normal[i]h(\kgens[i]) \right)
	\]
	for some $w'\in\commsub{\free \Delta}$ with a quadratic number of relations. By \cref{lemma:diagonal-combined}, we can replace $ \normal[j]g(\kgens[j]) \cdot \normal[i]h(\kgens[i]) $, where $ j > i $, with
	\[
		\left[\normal[j]g(\diaggens), \normal[i]h(\diaggens)\right] \cdot \normal[i]h(\kgens[i]) \cdot \normal[j]g(\kgens[j])
	\]
	and use again \cref{lemma:multiple-swaps} to move the commutator in the alphabet $ \diaggens $ to the beginning of the expression.

	This allows us to rearrange the factors: after $ \frac{(n-1)n}{2} $ of these operations, we obtain an expression of the form
	\[
		\normal[d]{} (\diaggens) \cdot \normal[1]g(\kgens[1]) \cdot \normal[1]h(\kgens[1]) \cdots \normal[r]g(\kgens[r]) \cdot \normal[r]h(\kgens[r])
	\]
	that represents the element $ gh \in \K $: by \cref{lem:norml-form-exists-unique} this is the normal form $ \normal{gh} $ of $gh$, so we conclude.

\end{proof}

\def\n{3}
\def\r{2}
\def\kgens{x_1,x_2,y_1,y_2}
\RenewDocumentCommand{\K}{O{\n} O{\r} O{}}{K^{#1}_{#2}(\ifblank{#3}{#2}{#3})}
\RenewDocumentCommand{\Fii}{m O{\r}}{F^{(#1)}_{#2}}
\RenewDocumentCommand{\dirprod}{O{\r} O{\n}}{%
	\ifthenelse{\equal{#2}{3}}{%
		\freecopy{#1}{a} \times \freecopy{#1}{b} \times \freecopy{#1}{c}%
	}{%
		\freecopy{#1}{1} \times \dots \times \freecopy{#1}{#2}%
	}%
}

\section{The Dehn function of the Bridson-Dison group}

\label{sec:disons-group}

We now focus our attention on the case $ n=3, r=2 $, so we consider the Bridson-Dison group $ \K $, defined as the kernel of $ F_2 \times F_2 \times F_2 \to \ZZ^2 $.

This group was studied in detail in \cite{Dison-08, Dison-09}, where it was proven that its Dehn function $ \delta_{\K} $ satisfies $ N^3 \preccurlyeq \Dehn\K \preccurlyeq N^6$; the lower bound was already proven independently by Bridson \cite{BridsonPersonal}. It is finitely presented by
\[
	\K = \presentation{x_1, x_2, y_1, y_2}{[x_1,y_1],  [x_2,y_2], [x_1^{\epsilon_1},y_2^{\epsilon_2}][x_2^{\epsilon_2}, y_1^{\epsilon_1}]},
\]
where $ \epsilon_1, \epsilon_2 \in \{\pm 1\} $.

The presentation above can either be obtained from the one described in \cite[Section 13.5]{Dison-08}%
, via the identification $ \alpha_1 \mapsto x_1,\, \alpha_2 \mapsto y_1,\, \beta_1 \mapsto x_2,\, \beta_2 \mapsto y_2 $, or by applying the results of \cite{UniformBounds-25}.

If we follow the same strategy as in \cite{UniformBounds-25}, we obtain that the area of the push is at most quartic, so we conclude that $ \Dehn{\K} \asympleq N^6 $, which is the same bound obtained via the pushing argument in \cite{Dison-08}. This bound is however not sharp: in what follows we prove that $ \Dehn\K \asymp N^4 $.

We consider three copies of the free group with two generators, denoted by $ \freecopy 2a, \freecopy 2b, \freecopy 2c $, with generators denoted by $ a_1, a_2; b_1, b_2; c_1, c_2 $ respectively. Then we take the homomorphism
\[
	\bigfun{\fibration}{\dirprod}{\ZZ^2}{a_i, b_i, c_i}{e_i.}
\]
where $e_1,e_2$ is the standard basis for $\ZZ^2$.
The inclusion $ \K \longhookrightarrow \dirprod$ is given by $ x_i \mapsto a_i\inv{c}_i$, $y_i \mapsto b_i\inv{c}_i$.

Let $\trivialwords$ denote the group of elements of $ \free{x_1,x_2,y_1,y_2} $, up to free equivalence, that represent the trivial element in $ \K $. In other words, we have an exact sequence
\[
	\shortexactsequence{\trivialwords}{\free{x_1,x_2,y_1,y_2}}{\K}[.]
\]
\begin{remark}[Detecting trivial words]\label{rmk:alg-for-trivial-words}
	Given a word $w$ in the letters $x_1,x_2,y_1,y_2$, there exists a simple algorithm to check whether $w$ represents the trivial element in $\K$: since $\K$ is a subgroup of $\dirprod$, the word $w$ is trivial if and only if the projection of $w$ on each of the three factors $ \freecopy 2a, \freecopy 2b$ and $\freecopy 2c$ is trivial.

	By definition, the projection of $w$ on $\freecopy 2a$ is obtained by sending $\y_1,\y_2$ to $1$ and $x_1, x_2$ to $a_1, a_2$ respectively; analogously, the projection of $w$ on $\freecopy 2b$ is obtained by sending $\x_1,\x_2$ to $1$ and $\y_1,\y_2$ to $b_1, b_2$ respectively. The projection of $w$ on $F_2^{(c)}$ is obtained by sending $\x_1,\y_1$ to $\inv c_1$ and $\x_2,\y_2$ to $\inv c_2$.

	Therefore, the word $w$ is trivial in $ \K $ if and only if for each of the following operations:
	\begin{itemize}
		\item deleting all occurrences of $x_1, x_2$ in $w$,
		\item deleting all occurrences of $y_1, y_2$ in $w$,
		\item replacing all occurrences of $ y_i $ with $ x_i $ (i.e.~we consider $x$ and $y$ as the same letter),
	\end{itemize}
	the resulting word is trivial as an element of $F_2$.
\end{remark}

\subsection{Upper bound}\label{sub:upper-bound-r2}
We start by proving the upper bound, i.e.~that $ \Dehn\K \asympleq N^4 $. The proof is similar to the one given for the general case, except that instead of using $ \dirprod $ as ambient group for the push-down argument, we use the well-known Stallings' group $ \SB $. Initially constructed by Stallings \cite{stallings}, it was later described by Bieri \cite{bieri} as the kernel of the morphism $ \dirprod \to \ZZ $, that sends every generator $ a_i, b_i, c_i $ to the generator $ 1 \in \ZZ $. It was proven in \cite{carter2017stallings} that the Dehn function of $ \SB $ is quadratic.

We have the natural inclusions given by the diagram
\[
	\begin{tikzcd}
		\K \arrow[rd, hook] \arrow[d, hook] \\
		\SB \arrow[r, hook] & \dirprod \arrow[rd]\arrow[r] & \ZZ  \\
		&& \ZZ^2 \arrow[u] \\
	\end{tikzcd}
\]
where the projection $ \ZZ^2 \to \ZZ $ sends the generators $ e_1 $ and $ e_2 $ to $1 \in \ZZ$.

The group $ \SB $ is generated by five elements $ x_1, x_2, y_1, y_2, s$, where the first four are the image of the homonymous generators of $ \K $. The inclusion $ \SB \hookrightarrow \dirprod $ is given by $ x_i \mapsto a_i \inv c_i $, $ y_i \mapsto b_i \inv c_i $, and $ s \mapsto \inv c_1 c_2 $.%

A finite presentation is given by
\begin{equation*}\label{eq:sb-group}
	\SB=
	\left\langle\x_1\,\x_2,\y_1,\y_2,\s\ \middle|\
	\begin{aligned}
		 & [\x_1,\s\y_2]=[\y_1,\s\x_2]=[\x_1,\y_1]=[\x_2,\y_2]=1 \\
		 & \inv x_1sx_1=\inv x_2sx_2=\inv y_1sy_1=\inv y_2sy_2
	\end{aligned}\right\rangle.
\end{equation*}
\begin{remark}
	The above presentation can be obtained from the presentation
	\[
		\presentation{a,b,c,d,e}{\inv bab = \inv cac = \inv dad = \inv eae, [c,d]=[d,b]=[e,c]=[e,b]=1}
	\]
	in \cite{baumslag} by identifying $ a \mapsto s, b \mapsto x_1, c \mapsto sx_2, d \mapsto y_1, e \mapsto sy_2$. A sketch of proof that this is a presentation for $\SB$ can be found in \cite{gersten-95}.
\end{remark}

We define the \emph{height function} to be the homomorphism
\[
	h\colon \free\SBgens\to \ZZ,
\]
defined by $h(\s)=1$, $h(\x_1)=h(\x_2)=h(\y_1)=h(\y_2)=0$. In other words, $h(w)$ counts the number of $s$ (with sign) appearing in a word representing $w\in \free\SBgens$. We have the following short exact sequence:
\[
	\shortexactsequence{\K}{\SB}[h]{\ZZ}[.]
\]

We define the push-down function
\[
	\push\colon \ZZ\times  \free\SBgens \to \free\kgens
\]
as the unique function such that for $k\in \mathbb{Z}$
\begin{align*}
	\push_k(\s)   & =1                                   ,  \\
	\push_k(\x_i) & =(\y_1\inv \y_2)^k\x_i(\y_2\inv\y_1)^k, \\
	\push_k(\y_i) & =(\x_1\inv \x_2)^k\y_i(\x_2\inv\x_1)^k,
\end{align*}
for $ i\in \set{1,2} $, and
\[
	\push_k(w_1w_2)=\push_k(w_1)\push_{k+h(w_1)}(w_2)
\]
for every $ w_1, w_2 \in \free\SBgens $.

\begin{lemma}
	The map $ \push_k $ is a well-defined push-down map as per \cref{def:push-down}.
\end{lemma}

\begin{proof}
	The first property holds by construction. The second one also follows easily by noting that $ \push_0(x_i) = x_i $ and $ \push_0(y_i) = y_i $. The well-definedness then follows as in \cref{lem:push-down-existence}.
\end{proof}

\begin{lemma}\label{lemma:c8-utile}

	Fix $ m \in \NN $. Then there is a constant $ C_m > 0 $ such that, for every $w, v \in F_2$ with $ \abs w, \abs v \leq m$, the word
	\[\left[v(y_1, y_2),w(x_1, x_2)^N\right]\overline{\left[v(x_1,x_2),w(y_1,y_2)^N\right]} \]
	represents the trivial element of $ \K $ and has area bounded by $ C_m \cdot N^2 $.
\end{lemma}

\begin{proof}

	By \cite[Lemma A.13]{UniformBounds-25}, we know that $ \Area([y_1, x_2^N]\inv{[x_1,y_2^N]}) \leq B N^2$ for some constant $ B>0 $.

	Define the homomorphism $ \phi \colon F(\afgen1,\afgen2) \mapsto F(\afgen1,\afgen2) $ by sending $ \afgen1 \mapsto v(\afgen1,\afgen2) $ and $ \afgen2 \mapsto w(\afgen1,\afgen2) $.  We now apply \cite[Proposition 3.7]{UniformBounds-25} to obtain that
    \[
        \Area( [v(y_1, y_2),w(x_1, x_2)^N]\overline{[v(x_1,x_2),w(y_1,y_2)^N]} ) \leq C_m \cdot N^2
    \]    
    for some constant $C_m$ depending only on $m$.

\end{proof}

\begin{lemma}\label{lemma:area-push-sb}
	There exists a constant $C$ such that the following holds.
	Let $R$ be a relation of $ \SB $. Then $ \push_N(R) $ is an element of $ \free\kgens $, whose area with respect to the presentation of $ \K $ is bounded by $ C N^2 $.
\end{lemma}
\begin{proof}%
	Let $R$ be as above. To prove the statement, we start from $ \push_N(R) $, and at each step we either multiply by at most $O(N^2)$ conjugates of some relations belonging to the presentation of $\K$, or we apply \cref{lemma:c8-utile}, until we get the trivial element. Both types of transformations involve a quadratic amount of relations, so this will give the desired conclusion.

	\begin{itemize}

		\item We start with $ \push_N([\x_1,\y_1]) $, which can be written as
		      \[
			      (\y_1\inv\y_2)^N\x_1(\y_2\inv\y_1)^N\cdot \y_1\left[\inv\y_1,(\x_1\inv\x_2)^N\right]
			      \cdot(\y_1\inv\y_2)^N\inv\x_1(\y_2\inv\y_1)^N\cdot \inv \y_1\left[\y_1,(\x_1\inv\x_2)^N\right].
		      \]
		      By applying \cref{lemma:c8-utile} to the two commutators we get
		      \[
			      (\y_1\inv\y_2)^N\x_1(\y_2\inv\y_1)^N\cdot \y_1\left[\inv x_1,( y_1\inv y_2)^N\right]
			      \cdot(\y_1\inv\y_2)^N\inv\x_1(\y_2\inv\y_1)^N\cdot \inv \y_1\left[ x_1,( y_1\inv y_2)^N\right].
		      \]
		      This simplifies to
		      \[
			      (\y_1\inv\y_2)^N\x_1(\y_2\inv\y_1)^N \y_1 \inv x_1 \inv \y_1 x_1 (y_1\inv y_2)^N \inv x_1 (y_2 \inv y_1)^N,
		      \]
		      which is a conjugate of $ [y_1, \inv x_1] $.

		\item Next consider $\push_N([x_1,sy_2])$. It is freely equivalent to

		      \begin{multline*}
			      \qquad (y_1\inv y_2)^N x_1(y_2\inv y_1)^N  (\x_1\inv\x_2)^{N+1} (y_2 \inv y_1) y_1 (x_2 \inv x_1)^{N+1}\cdot                                \\
			      \cdot ( y_1\inv y_2)^{N+1}\inv x_1( y_2\inv y_1)^{N+1}  ( x_1\inv x_2)^{N+1}\inv y_1 (y_1 \inv y_2)  (x_2 \inv x_1)^{N+1},
		      \end{multline*}
		      which, after commuting $ (y_2 \inv y_1) $ with $(x_1\inv x_2)$ a linear amount of times, becomes
		      \begin{multline*}
			      \qquad (y_1\inv y_2)^N x_1(y_2\inv y_1)^{N+1} (\x_1\inv\x_2)^{N+1}  y_1 (x_2 \inv x_1)^{N+1}                                \\
			      \cdot ( y_1\inv y_2)^{N+1}\inv x_1( y_2\inv y_1)^{N+1}  ( x_1\inv x_2)^{N+1}\inv y_1  (x_2 \inv x_1)^{N+1}(y_1 \inv y_2)
		      \end{multline*}
		      or, equivalently, \pagebreak[2]
		      \[
			      (y_2 \inv y_1) \cdot \push_{N+1}([x_1, y_1]) \cdot (y_1 \inv y_2)
		      \]
		      and we conclude as in the previous case.

		\item We now compute $\push_N(\inv\x_1\s\x_1\inv\y_1 \inv s y_1) $, which is equal to
		      \[
			      (\y_1\inv\y_2)^N\inv\x_1 (y_2 \inv y_1)^N (y_1 \inv y_2)^{N+1} x_1(y_2\inv y_1)^{N+1}\cdot (\x_1\inv\x_2)^{N+1}\inv y_1 (x_2 \inv x_1)^{N+1} (x_1 \inv x_2)^N y_1 (x_2 \inv x_1)^N.
		      \]
		      This can be rewritten as
		      \[
			      [(\y_1\inv\y_2)^N, \inv\x_1 ][\inv x_1, (y_1 \inv y_2)^{N+1}]\left[(\x_1\inv\x_2)^{N+1},\inv\y_1\right]\left[\inv\y_1,(\x_1\inv\x_2)^N\right].
		      \]
		      We apply \cref{lemma:c8-utile} twice to conclude.
		\item Finally, we consider $ \push_N(\inv\x_1\s\x_1\inv\x_2\ \inv s \x_2) $, which can be freely reduced to
		      \[
			      (\y_1\inv\y_2)^N\inv\x_1(\y_1\inv\y_2)\x_1 \inv\x_2 (\y_2\inv\y_1) \x_2(\y_2\inv\y_1)^{N}.
		      \]
		      This is conjugate to $ [y_1 \inv y_2, x_1 \inv x_2] $, so it has bounded area independent of $N$.
	\end{itemize}

	The proofs for all other relations are similar.
\end{proof}

We are now ready to prove the upper bound.

\begin{proposition}\label{prop:upper-bound-K322}
	The Dehn function of the group $ \K $ satisfies $ \Dehn\K \asympleq N^4 $.
\end{proposition}

\begin{proof}
	By \cref{lemma:area-push-sb}, the area of the push of the relations of $\SB$ is bounded by a quadratic function. Since $ \SB $ has quadratic Dehn function \cite[Cor. 4.3]{carter2017stallings}, it has linear radius by \cref{prop:linear-radius}. By applying \cref{thm:push-down} we thus get that $ \Dehn\K \asympleq N^4 $, as desired.
\end{proof}

\subsection{Lower bound}
We will now prove the lower bound. Our proof relies on a new invariant that we will call the braid-invariant. We will first introduce it and then show how it can be used to obtain a quartic lower bound on the Dehn function of $K_2^3(2)$.

\subsubsection{The braid-invariant}\label{sssec:braid-invariant}
We will associate to every $ w \in \trivialwords $ a loop inside the configuration space of two ordered points inside $ \CC \setminus \{0\} $, denoted by $ \Conf_2(\CC \setminus \{0\}) $. For this purpose, recall that the $n$-th configuration space $ \Conf_n(X) $ of a topological space $X$ is the subset of $X^n$ consisting of $n$-tuples whose coordinates are pairwise distinct:
\[\Conf_n(X)=\{(p_1,\dots ,p_n)\ |\ p_i\neq p_j \forall 1\leq i< j \leq n\}.\]
The space $\Conf_n(X)$ is equipped with the induced topology as a subset of the product $X^n$.

Now we restrict ourselves to the case of $X=\CC\setminus\{0\}$.
An element of the fundamental group $\pi_1(\Conf_2(\CZ), (\hat p_\x,\hat p_\y))$ can be described as a pair of paths $(\gamma_\x,\gamma_\y)$ in $\CZ$ in which $\gamma_\x(0)=\gamma_\x(1)=\hat p_\x$, $\gamma_\y(0)=\gamma_\y(1)=\hat p_\y$, and $\gamma_\x(t) \neq \gamma_\y(t)$ for all $t\in[0,1]$.
More intuitively, a representative of an element of the fundamental group $\pi_1(\Conf_2(\CZ), (\hat p_\x,\hat p_\y))$ can be thought of as two points moving inside $\CZ$ starting from the position $\hat p_\x,\hat p_\y$, without ever colliding and, at the end, coming back to $\hat p_\x,\hat p_\y$. The way we associate an element of $\pi_1(\Conf_2(\CZ))$ to a word $w \in \trivialwords $ is by prescribing how these two points should move; more precisely, the letters $\x_1,\x_2$ will command the first point to move right, respectively up by $1$ unit, while the letters $\y_1,\y_2$ will command the second point to move left, respectively down by $1$ unit; all the commands are processed one by one while reading the word $w$.

We will see that the commands given to the two points will make them come back to their starting position, because we are only considering words representing the trivial element in $\K$.
Moreover, if the starting positions are chosen suitably (that is by requiring that their coordinates and the coordinates of their difference are not integral), then, whatever the word $w$ is, the two points stay away from $0$ and do not collide.

In this way, given suitable starting positions, we define a homomorphism from the set of trivial words $\trivialwords$ to the fundamental group $\pi_1(\Conf_2(\CZ), (\hat p_\x,\hat p_\y))$. We use this map to define an invariant, that we call \emph{braid-invariant}, which we will use to provide a lower bound for the Dehn function of $\K$.

We will now formalize the above description.
Consider two Gaussian integers $p_\x, p_\y \in \ZZ[i]$ (i.e.~complex numbers with integral real and imaginary part), and set
\begin{align*}
	\hat p_\x\ =\   & p_x-\left(\frac{1}{3}+\frac{1}{3}i\right), \\
	\hat p_\y\ = \  & p_y+\left(\frac{1}{3}+\frac{1}{3}i\right),
\end{align*}
so that $\hat p_\x, \hat p_\y$ and $\hat p_\x-\hat p_y$ are complex numbers with neither coordinate being integral.

Denote $ \disonwords=F(x_1,x_2,y_1,y_2) $. Since a word $w\in \trivialwords$ represents the trivial element in $\tripr$, every generator appears in $w$ the same number of times as its inverse: indeed, $\x_i$ is the only generator that produces the letter $a_i$, and $\y_i$ is the only one that produces $b_i$. Therefore, $w$ is an element of $[\disonwords,\disonwords]$, and we obtain an injective homomorphism $ \trivialwords\hookrightarrow[\disonwords,\disonwords]$.

We define a homomorphism
\[\widetilde I_{p_\x,p_\y}\colon [\disonwords,\disonwords]\to \pi_1\left(\Conf_2(\CZ), (\hat p_\x,\hat p_\y)\right),\]
depending on $p_\x,p_\y$, that sends each element $g\in [\disonwords,\disonwords]$ to an element of the fundamental group $\pi_1(\Conf_2(\CZ),(\hat p_\x,\hat p_\y))$ represented by a closed loop $(\gamma_\x,\gamma_\y)$ in $\Conf_2(\CZ)$ defined as follows: let $w$ be a word representing $g$ and let $\ell$ be the number of letters appearing in $w$. Set $\gamma_\x(0)=\hat p_\x$, $\gamma_\y(0)=\hat p_\y$.
If the $k$-th letter of $w$ is $\x_1^\delta$ (resp. $\x_2^\delta$), for some $\delta\in \{\pm1\}$, then, the path $\gamma_\x|_{[(k-1)/\ell,k/\ell]}$ is the straight line from $\gamma_\x((k-1)/\ell)$ to $\gamma_\x((k-1)/\ell) + \delta$ (resp. $\gamma_\x((k-1)/\ell)+ i\cdot\delta$), while it is the trivial path if the $k$-th letter is $\y_1^\delta$ or $\y_2^\delta$.
Analogously, if the $k$-th letter of $w$ is $\y_1^\delta$ (resp. $\y_2^\delta$), then the path $\gamma_\y|_{[(k-1)/\ell,k/\ell}$ is the straight line from $\gamma_\y((k-1)/\ell)$ to $\gamma_\y((k-1)/\ell) -\delta$ (resp.~$\gamma_\y((k-1)/\ell)- i\cdot \delta$), while it is the trivial path if the $k$-th letter is $\x^\delta_1$ or $\x^\delta_2$.

\begin{lemma} The map $\widetilde I_{p_\x,p_\y}$ is well-defined for all Gaussian integers $p_x,p_y$.
\end{lemma}
\begin{proof}
	Let $g\in [\disonwords,\disonwords]$, $w$ a representative of $g$ and $(\gamma_\x,\gamma_\y)$ be as in the definition of $I_{p_\x,p_\y}$. Because none of the coordinates of the three complex numbers $\hat p_\x$, $\hat p_\y$ and $\hat p_\x - \hat p_\y$ is integral, and $\gamma_\x(k/\ell)$ (resp. $\gamma_\y(k/\ell)$) is obtained by adding to $\hat p_\x$ (resp. $\hat p_\y$) a Gaussian integer, the coordinates of $\gamma_\x(k/\ell)$, $\gamma_\y(k/\ell)$ and $\gamma_\x(k/\ell)-\gamma_\y(k/\ell)$ are also not integral for any $k\in\{1,\dots ,\ell\}$.
	As one of the two paths $\gamma_\x|_{[(k-1)/\ell,k/\ell]}$ and $\gamma_\y|_{[[(k-1)/\ell,k/\ell]}$ is constant, and the other is parallel to the real or to the imaginary axis, the three points $0$, $\gamma_\x(t)$ and $\gamma_\y(t)$ are pairwise distinct for every $t\in[0,1]$. This proves that $(\gamma_\x,\gamma_\y)$ is a path in $\Conf_2(\CZ)$.

	Moreover, since $w$ is a word representing an element in $[\disonwords,\disonwords]$, each generator appears the same number of times as its inverse. Thus, $\gamma_\x(1)=\gamma_\x(0)$, $\gamma_\y(1)=\gamma_\y(0)$ and $(\gamma_\x,\gamma_\y)$ is, indeed, a loop in $\Conf_2(\CZ)$ based at $(\hat p_\x, \hat p_\y)$.

	Finally, as the presentation of $\disonwords$ that we are considering has no relations, if two words represent the same element in $[\disonwords, \disonwords]$, then one can be obtained from the other by adding/removing a finite number of ``$q\overline q$'', where $q$ is a generator of the group. Then the corresponding loops are obtained from each other by removing/adding backtracking paths (paths that go in a direction and then go back right after), so without modifying the homotopy type of the loop. This proves that the map $\widetilde I_{p_\x,p_\y}$ does not depend on the chosen word $w$ representing $g$ and so the map is well-defined.
\end{proof}

We then define $I_{p_\x,p_\y}\colon \trivialwords\to \pi_1(\CZ)$ as the composition of the two maps $ \trivialwords\hookrightarrow[\disonwords,\disonwords]$ and $\widetilde I_{p_\x,p_\y}$.

Let $ q $ be one of the three letters $a,b,c$ and consider the following diagram

\[
	\begin{tikzcd}
		&[-2em] & & 1 \arrow[d] &[-2em] \\
		1 \arrow[r] & \trivialwords \arrow[r, hook]\arrow[d, "I_{p_\x,p_\y}"] & \commsub{\disonwords} \arrow[r]\arrow[d, "\phi^{(q)} "]\arrow[dl, "\widetilde{ I}_{p_\x,p_\y} "] & \smallgroup \arrow[r]\arrow[d, hook] & 1 \\
		& \pi_1(\Conf_2(\CZ),(\hat p_\x,\hat p_\y)) \arrow[d, "\psi_*^{(q)}"] & \commsub{F_2^{(q)}} \arrow[dl, "I_{p_\x,p_\y}^{(q)}"]\arrow[d, hook] & F_2^{(a)}  \times F_2^{(b)} \times F_2^{(c)} \arrow[d]\arrow[dl, twoheadrightarrow] & \\
		&  \pi_1(\CC \setminus 0, \psi^{(q)}(\hat p_\x,\hat p_\y)) & F_2^{(q)} & \ZZ^2 \arrow[d]\\
		&&&1,
	\end{tikzcd}
\]

where
\begin{itemize}
	\item the map $\psi^{(q)}_*\colon \pi_1(\Conf_2(\CZ),(\hat p_\x,\hat p_y))\to\pi_1(\CZ,\psi^{(q)}(\hat p_\x,\hat p_y))$ is induced by the projection $ \psi^{(q)}\colon\Conf_2(\CZ) \to \CZ $ defined  by
	      \[
		      \psi^{(a)}(p_\x, p_\y) = p_\x,\qquad  \psi^{(b)}(p_\x, p_\y) = -p_\y, \qquad  \psi^{(c)}(p_\x, p_\y) = p_\y-p_\x;
	      \]

	\item the map $ \phi^{(q)}\colon [\disonwords,\disonwords]\to [F_2^{(q)},F_2^{(q)}] $ is the restriction of the projection $ \disonwords \rightarrow F_2^{(a)} \times F_2^{(b)} \times F_2^{(c)} \twoheadrightarrow F_2^{(q)} $ to the commutator subgroup;

	\item the map $ I_{p_\x,p_\y}^{(q)}\colon \commsub{F_2^{(q)} } \to \pi_1(\CC \setminus \{0\},\psi^{(q)}(\hat p_\x,\hat p_\y)) $ is defined

	      similarly to $ I_{p_\x,p_\y} $: to a word representing an element in $[F_2^{(q)},F_2^{(q)}]$ we associate a loop $\gamma_w$ based at $\psi^{(q)}(\hat p_\x,\hat p_\y)$ and obtained by concatenating several paths, one for every letter in $w$; more precisely, the generator $ q_1$ ($\overline{q}_1$, $q_2$, resp. $\overline{q}_2$,) in $w$ corresponds to a straight line of length $1$ going right (left, up, resp.~down).
\end{itemize}
\begin{lemma}
	The diagram described above commutes.
\end{lemma}

\begin{proof}
	The left triangle and the right pentagon commute by definition. We only need to prove that the left square commutes. If $g$ is an element in $[\disonwords,\disonwords]$ represented by the word $w$, then both the element $\psi_*^{(q)}\circ \widetilde I_{p_\x,p_\y}(g)$ and $ I^{(q)}_{p_\x,p_\y}\circ \phi^{(q)}(g)$ are represented by loops in $\CC$ defined piecewise, and each piece correspond to a generator in the word $w$. Therefore, it is enough to check that every generator of $ \disonwords $ induces the same path inside the two loops.

	Let $(\gamma_x,\gamma_y)$ be a pair of loops in $\CC$ given as in the definition of $\widetilde I_{p_\x,p_\y}(g)$. Then the generator $ \x_1 $ corresponds to a segment in $\gamma_\x$ of length $1$, parallel to the real axis and with the same orientation, while it corresponds to a constant path in $\gamma_\y$. By composing with $ \psi^{(q)} $, this path is sent to a segment that goes from a point $ z\in \CC $ to the point $ z+1 $, $ z $, $ z - 1 $ respectively when $ q $ is $ a $, $ b $, or $c$.
	Since $ \x_1 = a_1\inv{c}_1 $, we have that $ \phi^{(a)} (\x_1) = a_1, \phi^{(b)} (\x_1) = 1, \phi^{(c)} (\x_1) = \inv{c}_1$. After composing with $ I_{p_\x,p_\y}^{(q)} $, we obtain that the corresponding path is a segment that goes from a point $ z $ to $ z+1 $, $ z $, $ z-1 $ respectively, in the same way as above.

	The check for $\x_2$ is completely analogous. For $\y_1$ and $\y_2$ it is also similar, but one has to be careful that both $ \widetilde{I}_{p_\x,p_\y} $ and $ \psi_*(q) $ invert the sign, so the composition of the two maps ends up having the same sign.
\end{proof}

\begin{lemma}\label{lemma:untie-braids}
	Let $ w \in \trivialwords $ be a word representing the trivial element of $ \smallgroup $. Then $ \psi_*^{(q)} (I_{p_\x,p_\y}(w)) = 1 $ for all Gaussian integers $p_\x,p_\y$.
\end{lemma}

\begin{proof}
	The image of $w$ inside $ F_2^{(a)} \times F_2^{(b)} \times F_2^{(c)} $ is trivial, since the top row is exact; as the diagram commutes $ \phi^{(q)}(w) $ is trivial. Using again that the diagram commutes, implies the assertion.
\end{proof}

Summing up, given two Gaussian integers $p_\x,p_y$, we can associate to every word $w\in \trivialwords$ an element $I_{p_\x,p_\y}(w)$ belonging to $ \pi_1(\Conf_2(\CZ),(\hat p_\x,\hat p_y))$. We are interested in understanding for which $p_\x,p_\y$ a word $w$ is sent to a (non)trivial element of $ \pi_1(\Conf_2(\CC \setminus 0),(\hat p_\x,\hat p_y)) $.

Recall that the fundamental group of $ \Conf_k(\CC) $ with base points $ p_1, \dots, p_k \in \CC $ is isomorphic to the pure braid group $ \PB_{k} $ on $k$ strands (see, e.g., \cite{Artin1950braid}).
\begin{remark}
	This isomorphism would be canonical if $ p_1 < \dots < p_k $ were lying on the real axis; otherwise, one would have to choose a path that moves the base points to the canonical ones. We can still think of elements of the fundamental group as pure braids, in the sense that they are naturally $k$-uples of strands in $ \CC \times [0,1] $, up to isotopy, where the $i$-th strand has fixed endpoints $ p_i \times 0 $ and $ p_i \times 1 $.
\end{remark}

We observe that $ \Conf_k(\CC \setminus \{0\})$ is a deformation retract of $ \Conf_{k+1}(\CC) $, so $ I_{p_x,p_y} $ defines a morphism
\[\trivialwords\to \pi_1(\Conf_3(\CC), (O, \hat p_\x, \hat p_\y)) \]
in which a word $w$ is sent to the pure braid on three strands based in $O,\hat p_\x,\hat p_\y$. The first strand is the constant one based in the origin $O$ of the complex plane $\C$; the other two strands are represented by $\gamma_\x,\gamma_y$ as in the definition of $I_{p_\x,p_\y}(w)$.

This can be summarized by the following commutative diagram:
\begin{equation*}\label{diag:braid-group}
	\begin{tikzcd}
		\trivialwords \arrow[r, "I_{p_\x,p_\y}"]     &\pi_1(\Conf_2(\CZ),(\hat p_\x,\hat p_y)) \arrow[r, "\cong"]\arrow[d,"\psi_*^{(q)}"]   &\pi_1(\Conf_3(\CC), (O, \hat p_\x, \hat p_\y))  \arrow[d, "\widetilde\psi^{(q)}_* "]   \\
		&\pi_1(\Conf_1(\CZ),\psi^{(q)}(\hat p_\x,\hat p_y)) \arrow[r, "\cong"]             &\pi_1(\Conf_2(\CC), \widetilde \psi^{(q)}(O, \hat p_\x, \hat p_\y))
	\end{tikzcd}
\end{equation*}
where the horizontal isomorphisms are given by adding the constant strand based at $O$, and $\widetilde \psi^{(q)} \colon \Conf_3 \to \Conf_2$ is the map that forgets about one of the three points: more precisely, $\widetilde\psi^{(a)}, \widetilde\psi^{(b)}, \widetilde\psi^{(c)}$ forgets about the second, third, and first point respectively.

Given a Gaussian integer $p\in \ZZ[i]$, we denote by $\norm{p}=\max\{\abs{\Re p}, \abs{\Im p}\}$.
\begin{lemma}\label{lemma:far-away-equal-untangled}
	Suppose $ w \in \trivialwords $ is a word of length $ \ell $ and let $p_\x,p_\y$ be two Gaussian integers. If $\max\{\norm{p_\x},\norm{p_y}\} > \ell$, then $ I_{p_\x,p_\y}(w) $ is trivial.
\end{lemma}

\begin{proof}
	We prove the statement only in the case $ \Re(p_x) > \ell $, being the other cases analogous.

	Let $\gamma_\x, \gamma_\y$ be the two loops in the definition of $ I_{p_\x,p_\y}(w)$. By hypothesis, these two loops have length at most $ \ell $. Thus
	\[
		\abs{\Re\left( \gamma_\x(t)-\gamma_\x\left(t'\right) \right)}, \abs{\Re\left( \gamma_\y(t)-\gamma_\y\left(t'\right) \right)} \le \ell/2
	\]
	for all $ t, t' $.
	Thus, we can find a vertical line between the origin and $ p_\x $, which  does not intersect any of $ \gamma_\x $ and $ \gamma_\y $.
	In particular, if we consider the braid of $\pi_1(\Conf_3(\CC), (O, \hat p_\x, \hat p_\y)) $ defined by $ I_{p_\x,p_\y}(w) $, then the strand based in $p_\y$ can be knotted with at most one of the other two strands, say with the one based in $O$. More precisely, we can isotope the braid such that the strand based in $ p_x $ is vertical, i.e.~$ \gamma_x $ is constant, and $ \gamma_y $ stays inside a (topological) ball that contains the origin and does not intersect $ \gamma_x $.

	However, the braid $\widetilde\psi^{(b)} (I_{p_\x,p_\y}(w))$, obtained by forgetting the  strand based in $p_\x$, is trivial by \cref{lemma:untie-braids}, implying that the strand based in $p_\y$ and the one based in $O$ are unknotted. So we conclude that $I_{p_x,p_y}(w)$ is trivial.
\end{proof}

We denote by $\theta_\x\colon \disonwords \to \ZZ[i]$ the homomorphism sending $ x_1 \mapsto 1$, $x_2 \mapsto i$, and $y_1,y_2 \mapsto 0 $. Similarly, we denote by $ \theta_\y\colon \disonwords \to \ZZ[i] $ the homomorphism that sends $y_1 \mapsto -1$, $y_2 \mapsto -i$, and $x_1,x_2 \mapsto 0$.

\begin{lemma}\label{lemma:conj}
	Let $ \alpha = \alpha(x_1,x_2,y_1,y_2) \in \disonwords$ be any word, $w \in \trivialwords$ be a word representing the trivial element of $ \smallgroup $, and $ p_\x,p_\y$ two Gaussian integers.
	Then $ I_{p_\x,p_\y}(w) $ is trivial if and only if $ I_{p_\x+\theta_\x(\alpha),p_\y+\theta_\y(\alpha)}(\inv\alpha w \alpha) $ is.
\end{lemma}

\begin{proof}
	By definition, one can obtain the two loops representing the element $I_{p_\x,p_\y}(w)$ and the element $ I_{p_\x+\theta_\x(\alpha),p_\y+\theta_\y(\alpha)}(\inv\alpha w \alpha) $ from each other by conjugating by some path. Thus, if one of the two is trivial, so is the other.
\end{proof}
\begin{definition}
	For every $ w \in \trivialwords $, define the \emph{braid-invariant} of $w$ as
	\[I(w) \coloneq \#\left\{(p_\x,p_\y) \in \ZZ[i]^2 \suchthat I_{p_\x,p_\y}(w) \text{ is nontrivial} \right\}.\]
\end{definition}

The following are useful properties of the braid-invariant.
\begin{lemma}\label{lemma:summability-braid-inv}
	If $w_1,w_2\in \trivialwords$ and $\alpha(\x_1,\x_2,\y_1,\y_2)$ is any word in $\disonwords$, then
	\begin{enumerate}
		\item  $\I(w_1)$ is finite;
		\item $\I(\inv \alpha w_1 \alpha)=\I(w_1)$;
		\item $\I(w_1w_2)\leq \I(w_1)+\I(w_2)$.
	\end{enumerate}
\end{lemma}
\begin{proof}
	The first item follows from \Cref{lemma:far-away-equal-untangled}, while the second is a direct consequence of \cref{lemma:conj}.
	The third item is due to the fact that $ I_{p_\x,p_\y}$ is a homomorphism: if $ I_{p_\x,p_\y}(w_1) $  and $ I_{p_\x,p_\y}(w_2) $ are trivial for some Gaussian integers $p_\x,p_\y$, then $ I_{p_\x,p_\y}(w_1w_2) $ is also trivial.
\end{proof}

\subsubsection{Braid invariant and Dehn function}\label{sssec:braids}
Let
\[
	w_n= [\x_1^n,\y_2^n][\x_2^n,\y_1^n].
\]
In this section we show that $\I(w_n)$ has quartic growth in $n$, while, clearly, the number of letters in $w_n$ grows linearly in $n$. By the end of this subsection we will show that this is enough to prove that the Dehn function of $\K$ is at least quartic (see \cref{prop:dehn=braid}).

In order to compute $ I_{p_x, p_y}(w_n) $ explicitly, we fix an isomorphism
\[
	\pi_1(\Conf_3(\CC), (O, \hat p_x, \hat p_y)) \isom \PB_3
\]
by projecting the braids, seen as paths inside $ \CC \times [0,1] $, to $ \RR \times [0,1] $, while keeping the information about over- and under-crossings.

Since $ \PB_3 $ is a finite-index subgroup of $ \rm B_3 $, we may canonically embed the fundamental group of $ \Conf_3(\CC) $ into $ \rm B_3 $, which has the following finite presentation:
\[\left\langle\sigma_1,\sigma_2 \ |\ \sigma_1\sigma_2\sigma_1= \sigma_2\sigma_1\sigma_2 \right\rangle,\]
where $ \sigma_i $ exchanges the $ i $-th and $ (i+1) $-th strand, as in \cref{fig:s1s2}.

\begin{figure}
	\centering
	\begin{subfigure}{0.45\textwidth}
		\centering
		\scalebox{-1}[1]{\includegraphics[height=3cm]{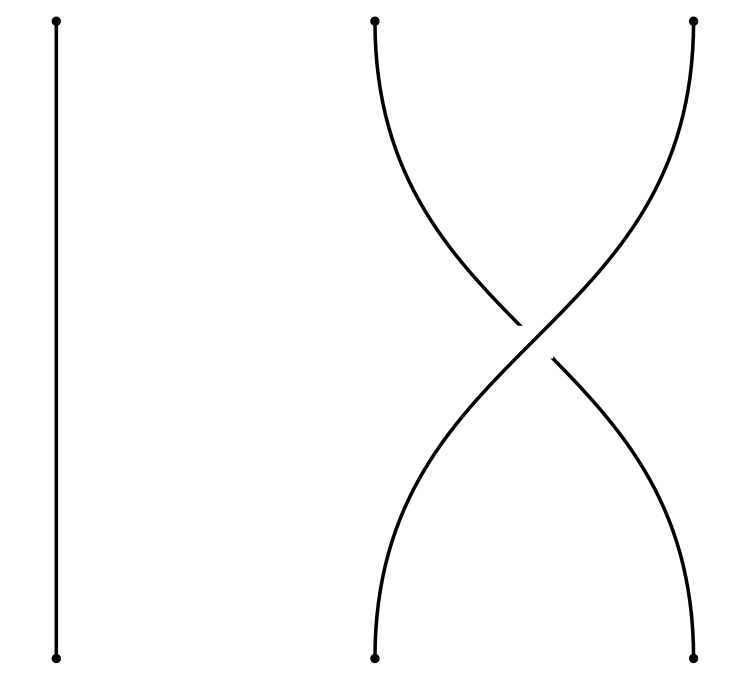}}
	\end{subfigure}
	\begin{subfigure}{0.45\textwidth}
		\centering
		\raisebox{\depth}{\scalebox{1}[-1]{\includegraphics[height=3cm]{pictures/braid-generator}}}
	\end{subfigure}
	\caption{Paths representing the elements $\sigma_1$ and $\sigma_2$.}
	\label{fig:s1s2}
\end{figure}

\begin{lemma}\label{lemma:points-of-nontriviality-for-wn}
	If
	\begin{equation*}
		\begin{cases}
			0\leq\Re({p_\y-p_\x}),\Im({p_\y-p_x})< n, \\
			\Re{(p_\x)},\Im{(p_\x)}\leq 0,            \\
			\Re{(p_\y)},\Im{(p_\y)}\geq 0,            \\
		\end{cases}
	\end{equation*}
	then, by using the notation just introduced for $\PB_3\subset \rm B_3$, we get that
	\[I_{p_\x,p_\y}(w_n)=\sigma_1\sigma_2^2\overline{\sigma}_1\sigma_2\overline{\sigma}_1^2\overline{\sigma}_2\]%
	is a nontrivial element of $\PB_3$.
\end{lemma}

\begin{proof}
	When computing $I_{p_\x,p_\y}(w_n)$ as described in \cref{sssec:braid-invariant}, one obtains the braid represented in \cref{fig:quartic-braid}, which coincides with the expression above.
	It is easy to check that this is a nontrivial element in $\rm B_3$: e.g., by closing the braid \cite[Chapter 1]{lickorish1997introduction}, one obtains a nontrivial link called \emph{Borromean rings} \cite{lindstrom1991borromean}, \cite{nanyes1993borromean}.
\end{proof}

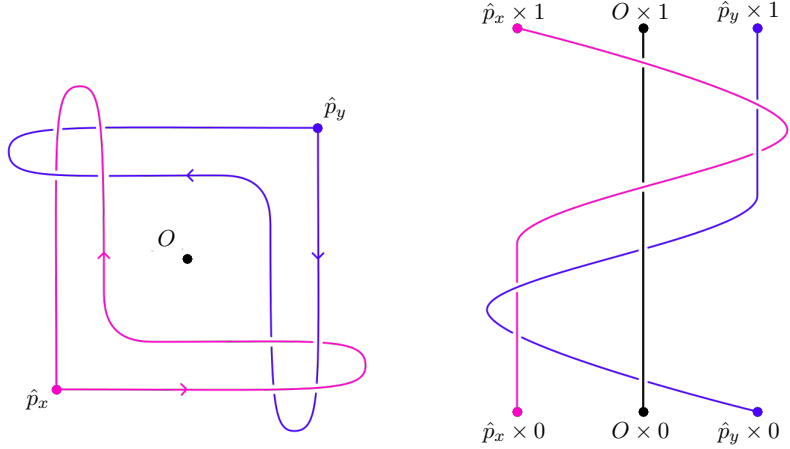
\begin{figure}
\scalebox{.8}{
	\begin{subfigure}{0.5\textwidth}
		\begin{tikzpicture}[scale = 1]
			\node[anchor=south west, inner sep=0] (image) at (0,0) {\includegraphics[height=7cm]{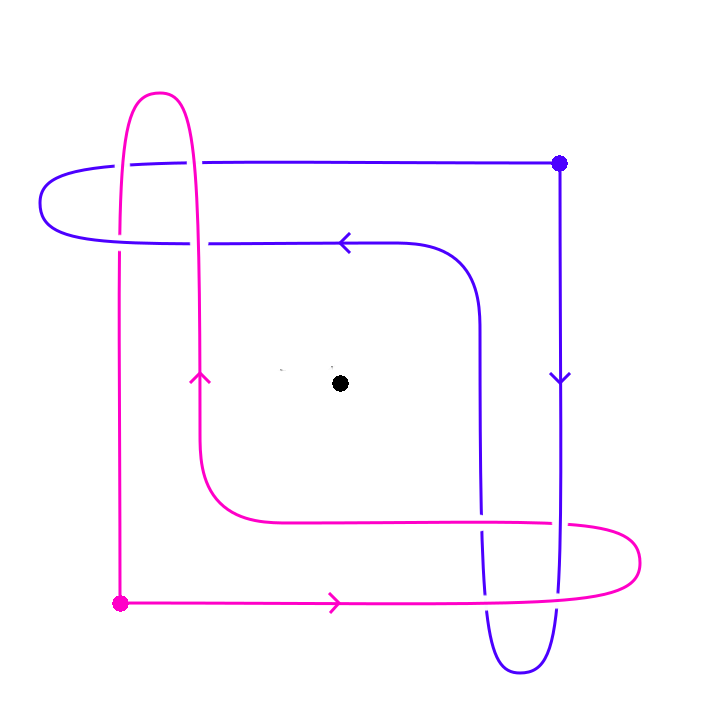}};
			\node[below left] at (1.2,1.2) {$\hat p_x$};
			\node[above right] at (5.5,5.4) {$\hat p_y$};
			\node[above left] at (3.3,3.3) {$O$};
		\end{tikzpicture}
	\end{subfigure}
	\begin{subfigure}{0.4\textwidth}
		\begin{tikzpicture}[scale=1]
			\node[anchor=south west, inner sep=0] (image) at (0,0) {\includegraphics[height=7cm]{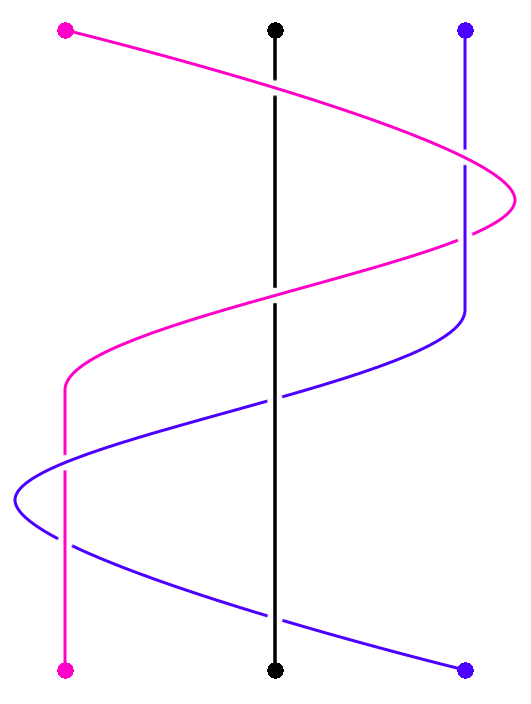}};
			\node[above] at (0.6,6.65) {$\hat p_x \times 1$};
			\node[above] at (2.7,6.7) {$O \times 1$};
			\node[above] at (4.5,6.65) {$\hat p_y \times 1$};
			\node[below] at (0.6,0.3) {$\hat p_x \times 0$};
			\node[below] at (2.7,0.3) {$O \times 0$};
			\node[below] at (4.5,0.3) {$\hat p_y \times 0$};
		\end{tikzpicture}
	\end{subfigure}
    }
	\caption{The braid $ I_{p_x, p_y}(w_n) $ inside $ \CC \times [0,1] $, seen from two different points of view: on the left, its projection to $ \CC $, where arrows denote the direction where the strands go down; on the right, the projection to $ \RR \times [0,1] $.}
	\label{fig:quartic-braid}
\end{figure}

Now we are ready to prove that the braid invariant constitutes a lower bound for the Dehn function of $\K$.
\begin{proposition}\label{prop:dehn=braid}
	Let $w \in \trivialwords$ be a word representing the trivial element of $ \smallgroup $. Then $ \Area(w) \geq C \cdot \I(w) $ for some constant $C>0$ depending only on the presentation.
\end{proposition}

\begin{proof}
	Set
	\[C' = \max\left\{\I(r) \ \left| \ r\in \trivialwords \textrm{ relation of } K^3_2(2)\right.\right\}.\]
	By \cref{lemma:summability-braid-inv}, the constant $C'$ is finite and only depends on the lengths of the relations.
	Using that $w_1$ is a relation of $\K$, and considering $p_\x=p_\y=0$ in \cref{lemma:points-of-nontriviality-for-wn}, one obtains that $I_{p_\x,p_\y}(w_1)$ is nontrivial and so both $I(w_1)$ and $C'$ are positive.

	If $w\in \trivialwords$ is any word representing the trivial word in $\K$, then, by applying the second and third item of \cref{lemma:summability-braid-inv}, we get
	\begin{align*} \I(w) & =
               \I\left(\prod_{i=1}^{\Area(w)} \inv{\alpha_i}r_i\alpha_i \right) \le
               \sum_{i=1}^{\Area(w)} \I(\inv{\alpha_i}r_i\alpha_i)=
               \sum_{i=1}^{\Area(w)} \I(r_i) \\
                     & \leq C'\cdot\Area(w),
	\end{align*}
	and one concludes by letting $ C = 1 / C' $.
\end{proof}

We are finally ready to prove that the Dehn function of $\K$ is at least quartic.

\begin{proposition}\label{prop:lower-bound-K322}
	The Dehn function of $\K$ is at least $N^4$.
\end{proposition}
\begin{proof}
	By \cref{lemma:points-of-nontriviality-for-wn}, whenever $p_\x,p_\y$ are two Gaussian integers such that
	\[-p_\x,p_\y\in [0,n/2)\times [0,n/2)\subset \CC,\]
	the element $I_{p_\x,p_\y}(w_n)$ is nontrivial in $\PB_3$.
	As there are at least $\left(\frac{n-1}{2}\right)^4$ pairs of Gaussian integers $(p_\x,p_\y)$ satisfying this condition, the invariant $\I(w_n)$ grows at least as a polynomial of degree $4$ in $n$. The assertion now follows from \cref{prop:dehn=braid}.
\end{proof}

\begin{proof}[Proof of \cref{main:disons-group}]
	The assertion is an immediate consequence of \cref{prop:upper-bound-K322,prop:lower-bound-K322}.
\end{proof}

Let us end by mentioning that as a consequence of \cref{prop:lower-bound-K322} and \cite[Lemma 4.3]{UniformBounds-25}, 
we also obtain lower bounds on the Dehn functions of $K_r^3(r)$ for $r\geq 2$. 
\begin{corollary}
    For $r\geq 2$ the Dehn function of $K_r^3(r)$ is bounded below by $N^4$.
\end{corollary}

\def\n{n}
\def\r{r}
\label{sec:braid}

\bibliographystyle{alpha}
\bibliography{references}

\end{document}